\documentclass[smallextended]{svjour3}       
\smartqed  
\usepackage{graphicx}
\usepackage{graphicx}
\graphicspath{{./img/}}	
\usepackage{xcolor}
\usepackage{lipsum}
\usepackage{amssymb}
\usepackage{amsmath}
\usepackage{mathrsfs}
\usepackage{hyperref}
\usepackage{mathtools}
\usepackage{dsfont}
\usepackage{url}
\usepackage{caption}
\usepackage{subcaption}
\usepackage{tcolorbox}
\usepackage{geometry}

\newcommand{\bm}{\boldsymbol}

\newcommand{\R}{\mathbb{R}}
\newcommand{\C}{\mathbb{C}}
\newcommand{\N}{\mathbb{N}}

\newcommand{\rev}[1]{#1}

\newtheorem{defn}{Definition}
\newtheorem{lem}{Lemma}
\newtheorem{ass}{Assumption}
\DeclarePairedDelimiter{\abs}{\lvert}{\rvert} 
\DeclarePairedDelimiter{\norm}{\lVert}{\rVert}
 
\newtheorem{cor}{Corollary}

\ifpdf
  \DeclareGraphicsExtensions{.eps,.pdf,.png,.jpg}
\else
  \DeclareGraphicsExtensions{.eps}
\fi

\usepackage{enumitem}
\setlist[enumerate]{leftmargin=.5in}
\setlist[itemize]{leftmargin=.5in}

\begin{document}

\title{Surrogate models for diffusion on graphs via sparse polynomials}
\author{Giuseppe Alessio D'Inverno\and Kylian Ajavon\and Simone Brugiapaglia}

\authorrunning{G.A. D'Inverno, S. Brugiapaglia, K. Ajavon} 

\institute{G. A. D'Inverno \at
              MathLab, International School for Advanced Studies \\
              Via Bonomea 265, Trieste, 34136, TS, Italy \\
              \email{gdinvern@sissa.it}           
           \and
           K Ajavon \& S. Brugiapaglia  \at
              Department of Mathematics and Statistics, Concordia University \\
              455 De Maisonneuve Blvd, Montréal, H3G 1M8, QC, Canada\\
              \email{kylian.yoan@gmail.com, simone.brugiapaglia@concordia.ca}
}
\date{Received: date / Accepted: date}

\maketitle

\begin{abstract}
Diffusion kernels over graphs have been widely utilized as effective tools in various applications due to their ability to accurately model the flow of information through nodes and edges. However, there is a notable gap in the literature regarding the development of surrogate models for diffusion processes on graphs. In this work, we fill this gap by proposing sparse polynomial-based surrogate models for parametric diffusion equations on graphs  with community structure. In tandem, we provide convergence guarantees for both least squares and compressed sensing-based approximations by showing the holomorphic regularity of parametric solutions to these diffusion equations. Our theoretical findings are accompanied by a series of numerical experiments conducted on both synthetic and real-world graphs that demonstrate the applicability of our methodology.
\keywords{diffusion on graphs \and sparse polynomial approximation \and least squares \and compressed sensing}
\subclass{34B45 \and 41A10 \and 41A63 \and 65D40}
\end{abstract}

\section{Introduction}
\noindent Over the last two decades, sparse polynomial approximation has emerged as one of the key tools to construct surrogate models, mainly motivated by applications to \emph{uncertainty quantification}; see, e.g., \rev{\cite{sparsepoly,cohen2015approximation,smith2024uncertainty,jakeman2017generalized}} and references therein. Given a high-dimensional parametric model, the idea is to approximate the parameter-to-solution map using a polynomial basis in the parametric space (\emph{cf.} the concept of \emph{polynomial chaos expansion} \cite{luthen2021sparse}). The main family of parametric models studied so far as parametric Partial Differential equations (PDEs), such as parametric elliptic equations or PDEs over parametrized domains; see, e.g., \cite{cohen2015approximation} and \cite[Chapter 4]{sparsepoly}. These PDE models are in fact known to lead to holomorphic parameter-to-solution maps and to fast (i.e., algebraic or exponential) polynomial approximation convergence rates \cite{cohen2011analytic}. This has been the key ingredient behind the success of these methods (especially those based on \emph{least squares} and \emph{compressed sensing}), which allows them to lessen the \emph{curse of dimensionality} \cite{sparsepoly}. The \emph{parametric diffusion equation} is one of the most well-studied parametric PDEs in the field. 

Yet, PDEs are not the only mathematical models employed to describe diffusion processes. A prominent alternative is provided by \emph{diffusion on graphs} \cite{kondor2002diffusion} (although there are well-known connections between the two areas \cite{calder2022lipschitz}\cite{coifman2006diffusion}\cite{garcia2020error}).  
Diffusion kernels have been proposed in their discrete variant over non-Euclidean manifolds such as graph manifolds \cite{belkin2003laplacian}\cite{coifman2005geometric}\cite{kondor2002diffusion}. The information flow (i.e.,  the derivative of the quantity of interest with respect to time) is modeled as a weighted exchange of information between each node and its neighbors, where the rule of exchange is determined by a specific kernel. In its simplest form, the information flow is modeled by a discrete heat equation, that involves the notion of \textit{graph Laplacian} \cite{smola2003kernels}. Diffusion on graphs is used as a modelling tool in a variety of applications \cite{dong2019learning}, such as marketing selection or influence prediction over social networks \cite{ma2008mining,rodriguez2014uncovering}, understanding information
flowing over a network of online media sources \cite{gomez2012inferring}, or predicting the spread of epidemics \cite{groendyke2011bayesian}. 
Diffusion kernels have recently gained interest by the graph deep learning community \cite{gasteiger2019diffusion}, leading to neural architectures in which the weights of the diffusion process are learnable \cite{chamberlain2021grand} or a certain quantity related to a specific diffusion kernel (e.g. gradient flows \cite{diunderstanding}) is minimized.

Despite the great success and popularity of diffusion on graphs, there seems to be a gap in the literature regarding the study of surrogate models for \emph{parametric} diffusion on graphs. This paper aims to bridge this gap. \rev{From a practical perspective, the proposed surrogate modelling strategy is most effective for problems where the diffusion coefficient varies with time, for which the full order model requires the use of a numerical time stepping methods, or when the underlying graph has large size. Our paper will study both scenarios.}

\subsection{Main contributions} 

\noindent The main contributions of our paper are in order:
\begin{enumerate}
\item We propose the construction of sparse polynomial-based surrogate models for parametric diffusion on graphs via \emph{least squares} and \emph{compressed sensing}.
\item We show that the parameter-to-solution map associated with parametric graph diffusion models is holomorphic (Theorem~\ref{thm:mainthm}). Leveraging this result,  we prove convergence theorems for least squares and compressed sensing based polynomial surrogate models (Corollaries~\ref{cor:LS} and ~\ref{cor:CS}). Corollaries~\ref{cor:LS} and \ref{cor:CS} show that it is possible to achieve quasi-optimal best $m$-term approximation rates using $m$ samples (up to log factors), hence showing that least squares and compressed sensing can lessen the curse of dimensionality for parametric graph diffusion.
\item We validate the proposed methodology by considering the specific case of diffusion on graphs with \emph{community structure} through extensive numerical experiments. First, we consider the case of synthetic graphs generated using the \emph{stochastic block model}. Then, we construct surrogate models for diffusion on a real-world graphs based on Twitter and Facebook data, respectively. \rev{In particular, we show that the use of surrogate models can offer speed-ups of several orders of magnitude with respect to the corresponding full order model while achieving good accuracy.}
\end{enumerate}

\subsection{Outline} 

\noindent  The paper is structured as follows. Section \ref{sec:surrogate_diffusion} begins by providing a brief overview on diffusion on graphs and main concepts of sparse polynomial approximation, and how those concepts are bridged in polynomial surrogate modeling for diffusion on graphs. In Section \ref{sec:main_results} we state our main results, i.e., theoretical guarantees for recovering an approximate solution to time-edge dependent diffusion equations on graph. Our theoretical findings are validated by several numerical experiments in Section \ref{sec:numerics}. Finally, in Section \ref{sec:conclusion}, we provide concluding remarks and outline potential avenues for future research.

\section{Surrogate diffusion on graphs}\label{sec:surrogate_diffusion}
\noindent In this section we introduce the main mathematical definitions and the key concepts needed to understand our theoretical findings, that will be presented and discussed in Section~\ref{sec:main_results}.

\subsection{Graph theory: basic concepts and notation}\label{subsec:graph_theory}
\noindent A \emph{graph} (or \emph{network}) $G = (V, E)$  is a mathematical object composed of a set of nodes (or vertices) $V$ and a set of edges $E$. An edge connecting two nodes $u,v \in V$ is denoted as a couple $(u,v) \in E$. The number of nodes is denoted as $|V|$. Here we focus on \emph{undirected graphs}, i.e., such that $(u,v) \in E$ if and only if $(v,u) \in E$. A graph can be represented by its adjacency matrix $\bm{A} \in \{0,1\}^{|V| \times |V|}$, where $A_{ij} = 1$ if $(i,j) \in E$,  $A_{ij} = 0$ otherwise. The \textit{degree of a node $i$} is defined as $\delta_i=\sum_{j=1}^{|V|} A_{ij}$, and the \textit{degree matrix} is the diagonal matrix $\bm{D} \in \mathbb{N}_0^{|V| \times |V|}$ defined by $D_{ii}=\delta_i$, where $\mathbb{N}_0 := \{0,1,2,3,\ldots\}$.  In this work we consider \emph{weighted} graphs, where the adjacency matrix is replaced by a weighted adjacency matrix $ \bm{W}  \in \mathbb{R}^{|V|\times |V|}$. Note that the unweighted adjacency matrix (and therefore the nodes' degree) can always be retrieved by a weighted adjacency matrix (namely, $A_{ij}=1$ if $W_{ij}\neq0$ and $A_{ij} =0$ otherwise). 
We consider networks where edge weights may change in time, i.e., $\bm{W}: [0, \infty ) \rightarrow \mathbb{R}^{|V|\times |V|}$. \rev{Throughout the paper we will consider only graphs satisfying this assumption.}

\begin{ass}
    $\bm{W}(t)$ is a continuous function of time.
\end{ass}

\noindent A central concept in this work is that of  \textit{community}. A community  $\mathcal{C}$ is a subset of nodes of a graph, i.e., $\mathcal{C}\subseteq V$, defined with respect to a measure of ``affinity''. For instance, in graphs generated by the Stochastic Block Model (SBM) \cite{holland1983stochastic}, this affinity is quantified by the number of connections of every node with the other ones in the community. Communities can also be defined through a community detection algorithm, as we will see in Section~\ref{sec:numerics}. Communities can induce a partition on the set of nodes, although it is not a strict requirement.

\subsection{Diffusion on graphs}

\noindent We start by defining diffusion processes on graphs. One could consider diffusion processes modeled by random walks, where a particle circulates randomly along the edges of a graph, or modeled by differential equations. In this work, we restrict our attention to the latter. The diffusion processes we are interested in can be considered as the discrete analogy to the diffusion of heat on a continuous domain. We want to be able to denote the quantity of a “substance” at each node $i \in V$ in a graph $G$, and how that quantity evolves as time progresses. We can denote this quantity by using a vector-valued function $ \bm{u} (t)$. For unweighted graphs, we obtain the following \emph{diffusion equation}:
\begin{equation}
\label{eq:diff}
\begin{cases}
\dot{\bm{u} }(t) = -c \bm{L} \bm{u} (t), & t \in [0,T]\\
\bm{u}(0) = \bm{u}_0 & \\
\end{cases},
\end{equation}
where $ \bm{L}  =  \bm{D} - \bm{A} $ is the (unweighted) graph Laplacian, $c \in \mathbb{C}$ is the (uniform) diffusion coefficient, $T>0$ is the final time and $\bm{u}_0$ is the initial condition. 
The solution, which can be derived in a few steps \cite{Newman2010}, is given by
$
 \bm{u} (t) = \sum_{i\in V} a_i(t)  \bm{v} _i,
$
where $\{\bm{v}_i\}_{i=1}^N$ are the eigenvectors of $\bm{L}$ and $\{a_i(t)\}_{i=1}^N$ are the (time-dependent) coefficients of the linear expansion in the basis $\{\bm{v}_i\}_{i=1}^N$.

For weighted graphs, we consider two variants of the diffusion equation. The first one is defined as
\begin{equation}
\label{eq:ed}
\begin{cases}
\dot{ \bm{u} }(t)= \bm{M}  \bm{u} (t), & t\in[0,T]
 \\
\bm{u}(0) = \bm{u}_0 & \\
\end{cases},
\end{equation}
where $ \bm{M} = \bm{C} \odot \bm{W} - \bm{D} ( \bm{C}, \bm{W} )$ (i.e., $ \bm{C }$ and $\bm{W} $ do not change in time), and $\bm{D}(\bm{C}, \bm{W})$ is a diagonal matrix defined as
\begin{equation} \label{eq:D}
D(\bm{C}, \bm{W})_{ij} = \delta_{ij} \sum_{k\in V} C_{ik} W_{ik}, \quad \forall i,j \in V,
\end{equation}
where $\bm{C} \in \mathbb{C}^{|V| \times |V|}$ is the diffusion coefficient matrix.  

Another variant allows for time dependence in $\bm{M}$, namely, 
\begin{equation}
\label{eq:ted}
\begin{cases}
\dot{ \bm{u} }(t)= \bm{M} (t) \bm{u} (t), & t\in[0,T] \\
\bm{u}(0) = \bm{u}_0 & \\
\end{cases},
\end{equation}
where 
\begin{equation}\label{eq:M}
    \bm{M} (t)= \bm{C} (t)\odot \bm{W} (t)- \bm{D} ( \bm{C} (t), \bm{W} (t)), \quad t \in [0,T],
    \end{equation}
with $\bm{C}(t)$ being the time-dependent diffusion coefficient matrix.

Problem \eqref{eq:ted} is well posed for any continuous $ \bm{M} (t)$, thanks to the Picard–Lindel\"of Theorem \cite[Theorem 6.I]{walter} since $ (t,\bm{u})  \mapsto  \bm{M} (t) \bm{u} $ is continuous and Lipschitz in $\bm{u}$ for any rectangular neighborhod of any arbitrary point $(\bar{t},  \bar{u})\in[0,T] \times \C^{|V|}$. In fact, $\| \bm{M} (t) \bm{u} -  \bm{M} (t) \bm{v} \|\leq \| \bm{M} (t)\|\cdot\| \bm{u}  -  \bm{v} \|$ and $\| \bm{M} (t)\|$ is uniformly bounded in $[0,T]$ since $ \bm{M} $ is continuous.

\subsection{Sparse polynomial approximation: basic concepts and methods}\label{subsec:sparse_pol_approx_concepts}

\noindent We proceed by introducing some basic notions of polynomial approximation in high dimensions, mostly following the presentation in \cite{sparsepoly}.  Consider a function $f:\mathcal{U}=[-1,1]^d\subseteq\R^d\to\C$, where $d>1$. Let $\varrho$ be a probability measure over $\mathcal{U}$ and $\{\psi_{\bm{\nu}} \}_{\bm{\nu}\in\mathbb{N}_0^d}$ a family of orthonormal polynomials over $L^2_{\varrho}(\mathcal{U})$, the space of square-summable functions on $\mathcal{U}$ with respect to the measure $\varrho$. In this setting, any function $f\in L_{\varrho}^2(\mathcal{U})$ has an $L^2_{\varrho}$-convergent orthonormal expansion
$
f = \sum_{ \bm{\nu} \in\N_0^d} c_{ \bm{\nu} } \psi_{ \bm{\nu} }.
$
The high-dimensional polynomial basis $\{ \psi_{\bm{\nu}} \}_{\bm{\nu} \in \N_0^d }$ is obtained as tensor product of such one-dimensional orthogonal polynomials $\{\phi_k\}_{k \in \mathbb{N}_0}$ over the interval $[-1,1]$, i.e., 
$\psi_{\bm{\nu}}(\bm{x}) = \prod_{k=1}^d \phi_{\nu_k}(x_k)$, for $\bm{x} = (x_1, \ldots, x_d)\in \mathcal{U}$. Furthermore, we use the \emph{multi-index} notation, i.e., instead of indexing a polynomial $\psi_i$ by $i\in\N$, we index it using a multi-index $ \bm{\nu} =(\nu_1,\nu_2,...,\nu_d)\in\N_0^d$.  
Typical examples of $\varrho$ and $\{\psi_{\bm{\nu}} \}_{\bm{\nu}\in\mathbb{N}_0^d}$ are the uniform measure and the \emph{multivariate Legendre polynomials}, respectively, defined by
\begin{equation}\label{eq:multi_Legendre}
    \psi_{\bm{\nu}}(\bm{y}) = \prod\limits_{k=1}^d \sqrt{2\nu_k+1}P_{\nu_k}(y_k), \quad \bm{y} \in [-1,1]^d, \quad \forall \bm{\nu}\in \N_0^d,
\end{equation}
where the $P_{\nu}$'s are the univariate Legendre polynomials defined, e.g., via \emph{Rodrigues's formula}: $P_{\nu}(y) = \frac{1}{2^\nu \nu!}\frac{d^\nu}{dy^\nu}(y^2-1)^\nu$, $y \in [-1,1]$, for all $\nu \in \N_0$.
Another popular choice for $\varrho$ and $\{\psi_{\bm{\nu}} \}_{\bm{\nu}\in\mathbb{N}_0^d}$ is given by the Chebyshev measure and Chebyshev polynomials, respectively; see \cite{sparsepoly} for further details.

Our goal is to find an $s$-sparse polynomial approximation to $f$. In other words, we want
$f\approx\sum_{ \bm{\nu} \in S}c_{ \bm{\nu} }\psi_{ \bm{\nu} }$, where $S\subseteq\N_0^d$ and $\abs{S}\leq s$.

Multi-index sets of interest for this work are the \emph{total degree} index set, defined by
\begin{equation}\label{eq:TD}
\Lambda_n^{\mathrm{TD}}:=\bigg\{ \bm{\nu} =\left(\nu_k\right)_{k=1}^d \in \mathbb{N}_0^d: \sum_{k=1}^d \nu_k \leq n\bigg\},
\end{equation}
and the \emph{hyperbolic cross} index set, defined by
\begin{equation}\label{eq:HC}
\Lambda_n^{\mathrm{HC}}:=\bigg\{ \bm{\nu} =\left(\nu_k\right)_{k=1}^d \in \mathbb{N}_0^d: \prod_{k=1}^d\left(\nu_k+1\right) \leq n+1\bigg\}.
\end{equation}
Note that the parameter $n\in\N_0$, called the \emph{order} of the index set, and the dimension $d$ affect the cardinality of the index set. There is a closed-form expression for the cardinality of the total degree index set \cite{sparsepoly}, given by
$
\abs{\Lambda_n^{\mathrm{TD}}}= {n+d \choose d}$, for all $d \in \N$, and $n \in \N_0$
The cardinality of the hyperbolic cross does not have an explicit formula, but it can be estimated as
$
\abs{\Lambda_{n-1}^{\mathrm{HC}}} \leq \min \{ 2 n^3 4^d, e n^{2+\log_2(d)} \}
$ (see \cite[Appendix B]{sparsepoly} and references therein).
In particular, $\abs{\Lambda_n^{\mathrm{HC}}} \ll \abs{\Lambda_n^{\mathrm{TD}}} $ for large $d$ or $n$. 
We illustrate these index sets in Figure \ref{fig:indexsets}, with $n=4$ and $d=3$.
\begin{figure}[t]
	\centering
	\includegraphics[width=0.35\linewidth]{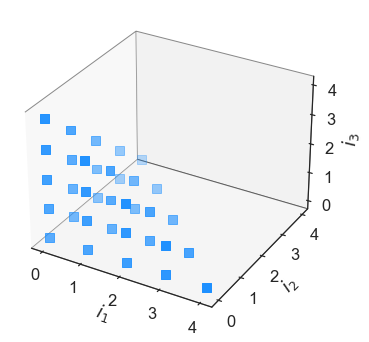}		
    \hspace{1cm}
    \includegraphics[width=0.33\linewidth]{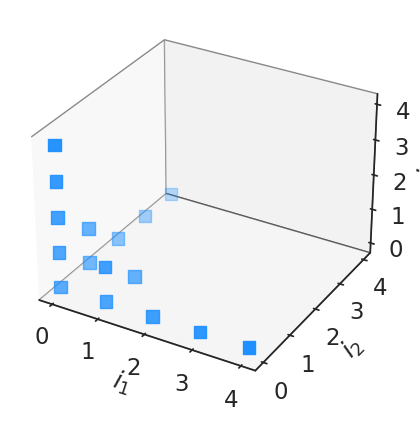}
	\caption{Multi-index sets of order $n=4$ in $\N_0^3$ (left: total degree; right: hyperbolic cross).} %
    \label{fig:indexsets}
\end{figure}
\vspace{5pt}

Different methods can be used to approximate $f$ from pointwise data depending on our knowledge of the set $S$. Here we focus on \emph{least squares} and \emph{compressed sensing}.

Suppose first we know $S=\{\bm{\nu}_1,...,\bm{\nu}_s\}$ a priori. We aim to find the coefficients $ \bm{c} =(c_{\bm{\nu}_1},...,c_{\bm{\nu}_s})\in\C^s$, so we consider the linear system $ \bm{\Psi}  \bm{c} = \bm{b} $, with $ \bm{\Psi} \in\C^{m\times s}$ and $ \bm{b} \in\C^m$, where the entries are defined as $\Psi_{ij}=\frac{1}{\sqrt{m}}\psi_{\bm{\nu}_j}(\bm{y}_i)$ and $b_i=\frac{1}{\sqrt{m}}(f(\bm{y}_i) + n_i)$ for $i=1, \dots, m$, with the points $\bm{y}_1, \dots , \bm{y}_m$ sampled independently from the domain $\mathcal{U}\subseteq\R^d$ according to the probability measure $\varrho$, and where $\bm{n}=(n_i)_{i=1}^m$ represents unknown noise corrupting the data. This leads to the following least-squares problem:
\begin{equation}\label{eq:ls}
\min_{ \bm{z} \in\R^s} \norm{ \bm{\Psi}  \bm{z} - \bm{b} }_2^2.
\end{equation}
An approximation $ \bm{\hat{c}} $ to the vector $ \bm{c} $ can be found using the Moore-Penrose pseudoinverse $\bm{\hat{c}} =( \bm{\Psi} ^* \bm{\Psi} )^{-1} \bm{\Psi} ^* \bm{b}$.
In order to retrieve a unique solution of a least-squares problem, we need an overdetermined system (a system with at least as many linearly independent equations as unknowns). In other words, for $ \bm{\Psi} \in\R^{m\times s}$, we need $m\geq s$. This means that we need more sample points than coefficients $ c_{\bm{\nu}_1}, \dots, c _{\bm{\nu}_s}$ for this method to be effective.
For further details on sparse polynomial approximation via least squares, we refer to \rev{the surveys in} \cite[Chapter 5]{sparsepoly} \rev{and \cite{cohen2015approximation}, and to notable seminal papers such as \cite{cohen2013stability,migliorati2013polynomial,migliorati2014analysis,chkifa2015discrete}}.

On the other hand, suppose we have no \emph{a priori} knowledge of $S$. In this case, a possible strategy is to choose a set $\Lambda\subseteq\N_0^d$ that contains $S$. A common choice of such $\Lambda$ in literature is represented by the hyperbolic cross, since it contains all so-called \textit{lower sets} of a given cardinality; see Appendix \ref{app:lower_sets} for further explanations).
If we consider a larger set $\Lambda$ that contains $S$, we come to a point where the system is underdetermined. In other words, for $ \bm{\Psi} \in\C^{m\times N}$, where $\Psi_{ij}=\frac{1}{\sqrt{m}}\psi_{\bm{\nu}_j}(\bm{x}_i)$ for all $i=1,..,m$, $j=1,...,N$, the cardinality $N = |\Lambda|$ of the set $\Lambda$ becomes greater than the number of sample points $m$. Therefore, if we want to recover the vector $\bm{c}\in\C^N$ of coefficients, this becomes a \emph{compressed sensing} problem \rev{\cite{donoho2006compressed,candes2006robust,foucart2013invitation}}, where we attempt to find the sparsest solution to the system $ \bm{\Psi}  \bm{c} = \bm{b} $. 
Popular methods to solve compressed sensing problems are the \emph{Square Root LASSO (SR-LASSO)} $\ell^1$-minimization problem, i.e.,
\begin{equation}\label{eq:sr-lasso}
    \min_{ \bm{z} \in\C^N} \lambda\norm{ \bm{z} }_1 + \norm{ \bm{\Psi} \bm{z} - \bm{b} }_2 \text {, }
\end{equation}
with $\lambda>0$, and the \emph{Quadratically Constrained Basis Pursuit} (QCBP) $\ell^1$-minimization problem, i.e.,
\begin{equation}
\label{eq:qcbp}
\min_{ \bm{z} \in\C^N} \norm{ \bm{z} }_1 \quad \text { subject to } \norm{ \bm{\Psi} \bm{z} - \bm{b} }_2\leq\eta \text {, }
\end{equation}
where $\eta>0$ and $\| \bm{z} \|_1 = \sum_{i=1}^N |z_i| $ is the $\ell^1 $-norm. Tipically, a good choice of $\eta$ satisfies $\norm{ \bm{n} }_2\leq\eta$. There are known techniques to solve equation \eqref{eq:qcbp}  efficiently, such as Chambolle and Pock’s Primal-Dual algorithm \cite{foucart2013invitation}. The SR-LASSO has the benefit that the optimal choice of the parameter $\lambda$ is independent of the noise $\bm{n}$ \cite{adcock2019correcting}. Note that the solutions to problems \eqref{eq:sr-lasso} or \eqref{eq:qcbp} are in general not sparse, but compressible (i.e., \emph{approximately} sparse). \rev{Notable early contributions on $\ell^1$ minimization for sparse polynomial approximations in high dimension include \cite{blatman2011adaptive,doostan2011non,mathelin2012compressed,rauhut2012sparse,yan2012stochastic}.}

To enhance the compressed sensing performance, a standard approach is to replace the $\ell^1$-norm by a \emph{weighted} $\ell^1$-norm. This approach relies on such weights to represent a form of prior on the unknown vector to be recovered. 
Given a vector of weights $\bm{w}\in \R^N$ with $\bm{w}\geq \bm{0}$, we can replace the $\ell^1$-norm with $\| \bm{z}\|_{1, \bm{w}} := \sum_{i=1}^N w_iz_i$ in equations \eqref{eq:sr-lasso} and \eqref{eq:qcbp}. This yields, respectively, the \emph{weighted SR-LASSO} problem
\begin{equation}\label{eq:wSR-LASSO}
    \min \limits_{ \bm{z} \in \mathbb{C}^N} \lambda \|  \bm{z}  \|_{1, \bm{w} } + \|  \boldsymbol{\Psi}  \bm{z}  - \bm{b}  \|_2
\end{equation}
and the \emph{weighted QCBP} problem
\begin{equation}
\label{eq:wqcbp}
\min_{ \bm{z} \in\C^N} \norm{ \bm{z} }_{1,\bm{w}} \quad \text { subject to } \norm{ \bm{\Psi} \bm{z} - \bm{b} }_2\leq\eta .
\end{equation}
For more information on the theoretical and practical benefits of weighted $\ell^1$-minimization we refer the reader to \cite{sparsepoly} \rev{and to key papers, such as \cite{adcock2019correcting,adcock2017infinite,chkifa2018polynomial,peng2014weighted,rauhut2016interpolation,yang2013reweighted}}.

High-dimensional functions may vary much faster in some coordinate directions than in others. This behavior is referred to as \emph{anisotropy}. For such functions, choices of $S$ that treat all coordinate directions equally usually lead to poor approximations \cite{sparsepoly}. Typically, least squares methods are preferred for known anisotropy, while compressed sensing is the standard option for unknown anisotropy. 
We will be using both least-squares and compressed sensing in our numerical experiments, and compare their performance in different scenarios.

Another important concept in sparse polynomial approximation is the \emph{best $s$-term approximation} of $f$. For $f = \sum_{ \bm{\nu} \in\N_0^d} c_{ \bm{\nu} } \psi_{ \bm{\nu} }\in L_{\varrho}^2(\mathcal{U})$, we define the \emph{best $s$-term approximation error} on its coefficient vector $\bm{c}$ according to a specific $\ell^p$ norm. 

\begin{defn}[Best $s$-term approximation error]\label{def:best_s-term}
    Let $0< p \leq \infty$, $ \bm{c} \in \ell^p(\mathbb{N}_0^d)$ and $s \in \N_0$. The $\ell^p$\emph{-norm best $s$-term approximation error} of $\bm{c}$ is defined as
\begin{equation}
    \sigma_s(\bm{c})_p := \min \big\{ \|\bm{c} - \bm{z} \|_p : \bm{z} \in \ell^p(\mathbb{N}_0^d), \; |\textnormal{supp}(\bm{z})|\leq s \big \}.
\end{equation}
where $\textnormal{supp}(\bm{z}) := \{ \bm{\nu} \in \mathbb{N}_0^d :z_{\bm{\nu}}\neq 0\}$,
\end{defn}

\subsection{Holomorphic regularity}\label{sec:holomorphic}
\noindent In the previous section, we have not discussed what type of functions $f$ can be accurately approximated by sparse polynomials. Holomorphic functions provide an answer to this question.

\begin{defn}[Holomorphic function]
\label{defn:holom}
Consider $f:\mathcal{O}\to\C$ where $\mathcal{O}\subseteq\C^d$ and $d\in\N$. The function $f$ is \emph{holomorphic} in $\mathcal{O}$ if the following limit exists for every $\bm{z} \in\mathcal{O}$ and every $j\in[d]$:
$
\lim_{h\in\C, h \to 0} \frac{f( \bm{z} +h \bm{e} _j) - f( \bm{z} )}{h},
$
where $ \bm{e} _j$ is the $j$-th element of the canonical basis of $\R^d$.
\end{defn}
\smallskip

\noindent Note that Definition \ref{defn:holom} still works if the codomain is $\C^n$ or $\C^{m\times n}$, for $m,n\in \N$. We now define the concept of filled-in Berstein polyellipse.

\begin{defn}[Filled-in Berstein polyellipse]
The \emph{filled-in Berstein polyellipse} of parameter $ \bm{\rho} =(\rho_j)_{j=1}^d\in\R^d$, with $\bm{\rho}>\bm{1}$ (that is, $\rho_j>1$ for $j=1,2,...,d$), denoted as $\mathcal{E}_{\bm{\rho}}$, is the Cartesian product of \emph{filled-in Berstein ellipses} of parameters $\rho_j$, defined as
$\mathcal{E}_{\rho_j} = \left\{ \frac{z+z^{-1}}{2}: z\in\C, 1\leq\abs{z}\leq \rho_j \right\}$.
Namely, $\mathcal{E}_{\bm{\rho}} = \mathcal{E}_{\rho_1} \times \mathcal{E}_{\rho_2} \times ... \times \mathcal{E}_{\rho_d}$.
\end{defn}
\smallskip

\noindent Assume that $f\in L_{\varrho}^2(\mathcal{U})$ can be extended in a holomorphic way to an open set $\mathcal{O}\subseteq\C^d$ such that $\mathcal{E}_\rho\subset\mathcal{O}$ for some $ \bm{\rho} > \bm{1}$, i.e., there exists $\tilde{f}: \mathcal{O}\to \C$ holomorphic such that $\tilde{f}|_{\mathcal{U}}=f $. Then, the best $s$-term approximation error $\sigma_s(\bm{c} )_q$ decays algebraically fast in $s$, with dimension-independent decay rate \cite[Theorem~3.6]{sparsepoly}. Namely,
\begin{equation}
\label{eq:bounderror}
\sigma_s( \bm{c} )_q \leq \frac{\norm{f}_{L^{\infty}(\mathcal{\mathcal{E}}_{\bm{\rho}})}\cdot C(d,p, \bm{\rho} )}{(s+1)^{\frac{1}{p}-\frac{1}{q}}},
\end{equation}
for all $0<p\leq q \leq \infty$ and $s\in\N_0$, where $ \bm{c} =(c_{ \bm{\nu} })_{ \bm{\nu} \in\N_0^d}$ is the infinite vector of coefficients of $f$, $C(d,p, \bm{\rho} )>0$ is a constant depending on $d$, $p$ and $ \bm{\rho} $ explicitly defined in \cite[Theorem~3.6]{sparsepoly}, and $L^{\infty}(\mathcal{\mathcal{E}}_{\bm{\rho}})$ is the space of essentially bounded functions $f:\mathcal{E}_{\bm{\rho}}\to\C$. 

Furthermore, with the holomorphic assumption on $f$, it is possible to show quasi-optimal convergence guarantees of sparse polynomial approximation via least squares or compressed sensing. Indeed, denoting with $\hat{f}_{\mathsf{LS}}$ a solution to problem \eqref{eq:ls} for a suitable polynomial space $S$, we have 
$$
\norm{f-\hat{f}_{\mathsf{LS}}}_{L_\varrho^2(\mathcal{U})} \leq C \cdot (m/\log(m/\epsilon))^{1/2-1/p},
$$
with probability at least $1-\epsilon$; see \cite[Theorem 6.1]{adcock2022monte} and Appendix~\ref{app:proofCorLS}. Similarly, denoting with $\hat{f}_{\mathsf{CS}}$ a solution to problem \eqref{eq:sr-lasso} respectively, we have
$$
\norm{f-\hat{f}_{\mathsf{CS}}}_{L_\varrho^2(\mathcal{U})} \leq \norm{f}_{L^{\infty}(\mathcal{E}_{\bm{\rho}})} \cdot C \cdot \Tilde{m}^{1/2-1/p},
$$
with probability at least $1-\epsilon$ and where 
$\tilde{m}=m / (c\cdot \tilde{L})$. Here  $\tilde{L}$ is a polylogarithmic factor that will be defined in Corollary~\ref{cor:CS} and
$c>0$ is a universal constant. Consequently, least squares and compressed sensing approximation achieve the same algebraic rate of convergence as that of the best $s$-term approximation in terms of $m$, up to log factors \cite[Theorem~7.12]{sparsepoly}.

\subsection{Community structure assumption for diffusion on graphs}

\noindent Consider a weighted graph $G=(V, E)$ with weighted adjacency matrix $\bm{W}$ of bounded size (i.e., $\abs{V}<\infty$) and with $K$ communities $\mathcal{C}_1, \mathcal{C}_2, ..., \mathcal{C}_K$. In the time-and-edge-dependent diffusion case, we consider the parametric map 
\begin{equation}
\label{eq:parametricmap}
\bm{C}:[0,T] \times  \mathcal{U}  \to \C^{|V| \times |V|}
\end{equation}
that describes how the (undirected) diffusion coefficients of the edges change over the parametric domain $\mathcal{U} \subset \mathbb{R}^{d}$ for fixed $t$. We assume $\mathcal{U}$ to be a compact hyperrectangle and without loss of generality, up to change of coordinates, we can assume that $\mathcal{U} = [-1,1]^d$. 

In our setup, we assign the same diffusion coefficient to the set of edges linking a pair of communities ($\mathcal{C}_i$, $\mathcal{C}_j$), for $i,j=1,2,...,K$. Thus, for a parameter $ \bm{y} =(y_1,...,y_d)$, we assign one entry $y_k$ of the parameter to the set of edges between the communities denoted by the pair  ($\mathcal{C}_i$, $\mathcal{C}_j$) and, by symmetry, the pair ($\mathcal{C}_j$, $\mathcal{C}_i$). As a result, we have the relation 
\begin{equation}\label{eq:def_d}
    d=\frac{K(K+1)}{2}.
\end{equation}
linking the dimension $d$ with the number of communities $K$.
The resulting image is a symmetric block matrix $\bm{C} (t)$, where each submatrix $ C_{ij}$ (or $ C(t)_{ij}$) gives the diffusion coefficients along all the edges between the communities $\mathcal{C}_i$ and $\mathcal{C}_j$. To provide some intuition about the community-based structure of the diffusivity map, we can think about, e.g., consensus dynamics in a political network (e.g., a parliament). In that context it is natural to expect that information will flow differently (i.e., faster or slower) within each party or across parties.

\rev{
\begin{remark}[Interpretation of the parameter $\bm{y}$] Throughout the paper we will interpret $\bm{y} \in [-1,1]^d$ as a deterministic vector defining the parametric diffusivity matrix $\bm{y} \mapsto \bm{C}(\cdot, \bm{y})$. However, it is worth pointing out that in applications involving uncertainty quantification \cite{smith2024uncertainty} the parameter $\bm{y}$ is typically a random vector modelling volatility in the diffusivity matrix $\bm{C}(\cdot, \bm{y})$ (which, in that scenario, would also be a random variable). 
\end{remark}
}

Additionally, we require that the map $\bm{C}$ has an holomorphic extension over a compact set $\mathcal{K}$. This requirement is essential to obtain recovery guarantees for the parametric diffusion map via least squares and compressed sensing (Section \ref{subsec:LS_CS}).

The following Assumption, that will be used in the main result of this paper (Theorem \ref{thm:mainthm}), summarizes the requirements just discussed.

\begin{ass}[Holomorphic extension]\label{ass:diffusion_matrix}
There exists a compact set $\mathcal{K}\subset \mathbb{C}^d$, with $\mathcal{U} \subset \mathcal{K}$,  and a map $\tilde{\bm{C}} : [0,T] \times \mathcal{K} \to \mathbb{C}^{|V| \times |V|}$ such that $\tilde{\bm{C}}|_{[0,T] \times \mathcal{U}} = \bm{C}$, $\tilde{\bm{C}}$ is continuous in $[0,T] \times \mathcal{K}$ and $\tilde{\bm{C}}(t, \cdot)$ is holomorphic in $\mathring{\mathcal{K}}$ for every $t \in [0,T]$. Moreover, we consider $\bm{C}$ of the form
$$
\bm{C}(t,\bm{y}) =
\begin{pmatrix}
    \bm{C}^{(1,1)}(t, \bm{y}) & \cdots & \bm{C}^{(1,K)}(t,\bm{y})\\
    \vdots & \ddots & \vdots\\
    \bm{C}^{(K,1)}(t,\bm{y}) & \cdots & \bm{C}^{(K,K)}(t,\bm{y})\\
\end{pmatrix},
$$
where $\bm{C}^{(i,j)}(t,\bm{y}) \in \mathbb{C}^{|\mathcal{C}_i| \times |\mathcal{C}_j|}$  satisfies $\bm{C}^{(i,j)}(t, \bm{y}) = c^{(i,j)}(t,\bm{y}) \cdot \mathds{1}$ with $\mathds{1}$ being an $|\mathcal{C}_i|\times |\mathcal{C}_j|$ matrix of 1's and such that $c^{(i,j)}(t,\bm{y}) = c^{(j,i)}(t,\bm{y})$ for all $i,j \in [K]$. In particular, $\bm{C}(t,\bm{y})$ is symmetric. The number of distinct functions $c^{(i,j)}(t,\bm{y})$ is $d$ as in \eqref{eq:def_d}.
\end{ass}
\rev{Whereas the presence of a community structure in the graph and the continuity of the diffusivity matrix with respect to the time variable are  intrinsic features of the problem, the specific form of the parametrization $\bm{y}\mapsto \bm{C}(\cdot, \bm{y})$ can be designed by the user. Hence,  Assumption~\ref{ass:diffusion_matrix} should be thought of as a guiding principle to be followed while parametrizing the family of problems of interest.} We now provide a typical example of parametrization satisfying this assumption. \rev{In the numerical experiments of Section~\ref{sec:numerics} only parametric diffusivity matrices of this form will be considered.}

\begin{example}\label{example:diffusion_matrix}
Given a parameter $ \bm{y} =(y_1,y_2,...,y_d)\in [-1,1]^d$, with $d$ as in \eqref{eq:def_d}, the matrix $ \bm{C}( t, \bm{y})$ in equation \eqref{eq:ted} can be constructed using an \textit{affine parametrization}, e.g.,
\begin{equation}
\label{eq:Ct}
\bm{C}(t, \bm{y}) =
\begin{pmatrix}
    \frac{y_{\sigma(1,1)}+1}{2}\cdot h_{\sigma(1,1)}(t)\mathds{1} & \cdots & \frac{y_{\sigma(1,K)}+1}{2}\cdot h_{\sigma(1,K)}(t)\mathds{1}\\
    \vdots & \ddots & \vdots\\
    \frac{y_{\sigma(K,1)}+1}{2}\cdot h_{\sigma(K,1)}(t)\mathds{1} & \cdots & \frac{y_{\sigma(K,K)}+1}{2}\cdot h_{\sigma(K,K)}(t)\mathds{1}\\
\end{pmatrix},
\end{equation}
where in the $(i,j)$-th block,  $\mathds{1}\in\R^{\abs{\mathcal{C}_i}\times\abs{\mathcal{C}_j}}$ denotes a matrix of ones of suitable dimension, $\sigma: [K] \times [K] \mapsto [d]$ is a function that defines a unique ordering (e.g., lexicographic) on $\{(i,j) \in [K]^2 : i \leq j\}$ and with $\sigma(i,j) = \sigma (j,i)$ for $i > j$. The functions $h_1(t),...,h_d(t)$ are continuous on  $[0,T]$. 
Affine parametrizations are very popular in modeling parametric PDEs \cite{sparsepoly,cohen2015approximation} and represent a simple way to model diffusivity parameters ranging in some prescribed interval. 
\end{example}

\section{Theoretical guarantees}\label{sec:main_results}

\noindent We are now in a position to state the main theoretical results of the paper. In the case of diffusion on graphs, our goal is to construct a polynomial that approximates the parameter-to-solution map $f_v:\mathcal{U}=[-1,1]^d\to\C$ defined as
\begin{equation}
\label{eq:goal}
f( \bm{y} )=u_v(T,\bm{C}( \cdot, \bm{y} ))
\end{equation}
for a given node $v$, where $\bm{u}(\cdot,\bm{C}(\cdot, \bm{y}))$ is the unique solution to the problem \eqref{eq:ted} with $\bm{C}(t) = \bm{C}(t,\bm{y})$. As illustrated in Section \ref{sec:holomorphic}, functions that can be extended in a holomorphic way have desirable properties with respect to the recoverability of their polynomial approximation; hence our goal is to show that there exists a holomorphic extension for $f$. We will use \cite{walter} and \cite[Lemma~4.3]{sparsepoly} which states the following:

\begin{lem}[Holomorphic dependence; Volterra integral equations]
\label{lem:holo}
Let $T>0$, $n,d \in \N$, $\mathcal{K}\subset\C^d$ be a compact set, and $ \bm{g} :[0,T]\times\mathcal{K}\to\C^n$ and $ \bm{h} :[0,T]^2\times\C^n\times\mathcal{K}\to\C^n$ be continuous functions. Moreover, suppose that there exists a constant $L>0$ such that $ \bm{h} $ satisfies the Lipschitz condition
\begin{equation}
\label{eq:lemma_Lip_cond}
\norm{ \bm{h} (t,s, \bm{u} , \bm{z} ) -  \bm{h} (t,s, \bm{v} , \bm{z} )} \leq L\norm{ \bm{u} - \bm{v} }, \quad \forall t,s\in[0,T],  \quad \forall  \bm{u} , \bm{v} \in\C^n, \forall  \bm{z} \in\mathcal{K},
\end{equation}
for some norm $\norm{\cdot}$ over $\C^n$. Then the Volterra integral equation
$$
 \bm{u} (t, \bm{z} ) =  \bm{g} (t, \bm{z} ) + \int_{0}^{t}  \bm{h} (t,s, \bm{u} (s, \bm{z} ), \bm{z} ) ds, \quad  \forall t\in[0,T],
$$
admits a unique solution $t\mapsto \bm{u} (t, \bm{z} )$ for every $ \bm{z} \in\mathcal{K}$ and the mapping $(t, \bm{z} )\mapsto \bm{u} (t, \bm{z} )$ is continuous in $[0,T]\times\mathcal{K}$. In addition, let $\mathring{\mathcal{K}}$ be the interior of $\mathcal{K}$ and assume that $ \bm{z} \mapsto \bm{g} (t, \bm{z} )$ is holomorphic in $\mathring{\mathcal{K}}$ for any fixed $t\in[0,T]$ and $( \bm{v} , \bm{z} )\mapsto \bm{h} (t,s, \bm{v} , \bm{z} )$ is holomorphic in $\C^n\times\mathring{\mathcal{K}}$ for any fixed $(t,s)\in[0,T]^2$. Then $ \bm{z} \mapsto \bm{u} (t, \bm{z} )$ is holomorphic in $\mathring{\mathcal{K}}$ for any fixed $t\in[0,T]$.  
\end{lem}

\rev{
\begin{remark}[Choice of the vector norm $\|\cdot\|$] The Lipschitz condition \eqref{eq:lemma_Lip_cond} of Lemma~\ref{lem:holo} depends on the choice of vector norm $\|\cdot\|$. We observe that this choice has no practical implications on the conclusion of the lemma. In the proof of Theorem~\ref{thm:mainthm}, which will crucially rely on Lemma~\ref{lem:holo}, we will choose $\|\cdot\| = \|\cdot\|_\infty$ because this simplifies computations, but other equally valid choices could be made.
\end{remark}
}

\noindent This lemma implies the existence of a holomorphic extension for the general setup
$$
\bm{f}( \bm{y} )= \bm{u}(T,\bm{C}( \cdot,\bm{y} )), \quad\forall \bm{y} \in\mathcal{U}, \quad T>0,
$$ 

\noindent where $\bm{u}(\cdot,\bm{C}(\cdot, \bm{y}))$ is the unique solution to the problem \eqref{eq:ted} with $\bm{C}(t) = \bm{C}(t,\bm{y})$. Note that this  immediately implies the existence of a holomorphic extension for $f_v(\bm{y}) = u_v(T, \bm{C}(\cdot,\bm{y}))$. This leads us to our main theorem.

\begin{theorem}[Holomorphy of parametric graph diffusion]
\label{thm:mainthm}
Let $T>0$ and let the parametric diffusivity matrix $ \bm{C}:[0,T] \times \R^{|V| \times |V|}$ satisfy Assumption \ref{ass:diffusion_matrix} for some compact $\mathcal{K}\subseteq\C^d$ such that $ \mathcal{U}  \subset \mathcal{K}$.
Then the map $\bm{f}:\mathcal{U}=[-1,1]^d\to\C^{|V|}$ defined by $\bm{f}( \bm{y} )= \bm{u}(T, \bm{y} )$ satisfying problem \eqref{eq:ted} admits a holomorphic extension $\tilde{\bm{f}}$ to $\mathring{\mathcal{K}}$. Moreover, for any compact subset $\mathcal{H} \subset \mathring{\mathcal{K}}$, we have
\begin{equation}
\label{eq:unif_upper_bound_ext}
\rev{\|\tilde{\bm{f}}\|_{L^{\infty}(\mathcal{H})}:=\sup_{\bm{z} \in \mathcal{H}} \|\tilde{\bm{f}}(\bm{z})\|_2} \leq  B(\bm{u}_0, T, \rev{\tilde{\bm{M}}}, \rev{\mathcal{H}}),
\end{equation}
where 
\begin{equation}\label{eq:bound_term}
   B(\bm{u}_0, T, \rev{\tilde{\bm{M}}}, \rev{\mathcal{H}}) =  \| \bm{u}_0\|_2 \cdot \exp\left(2\cdot \int_0^T \sup_{ \bm{z}  \in \mathcal{H} } \| \rev{\tilde{\bm{M}}} (t, \bm{z} )\|_{2\to 2} \, d t\right) < \infty,
\end{equation}
\rev{where $\tilde{\bm{M}} (t, \cdot ) := \tilde{\bm{C}}( t,\cdot) \odot  \bm{W} (t) -  \bm{D} ( \tilde{\bm{C}}( t,\cdot), \bm{W}(t))$ is a holomorphic extension of the function $\bm{M}(t, \cdot)$ defined in \eqref{eq:M} to $\mathcal{K}$ for every $t \in [0,T]$, and $\tilde{\bm{C}}$ is as in Assumption~\ref{ass:diffusion_matrix}}.
\end{theorem}

\begin{proof}
The proof is organized in two main steps. First, we prove the existence of a holomorphic extension to $\bm{f}$ (Step 1). Then, we show the validity of the uniform upper bound \eqref{eq:unif_upper_bound_ext} (Step 2).

\paragraph{Step 1: existence of a holomorphic extension}  We define a holomorphic extension $\tilde{\bm{f}}$ in a natural way using the complex-valued extension $\tilde{\bm{C}}$ to $\bm{C}$, whose existence is guaranteed by Assumption~\ref{ass:diffusion_matrix}. Now, observe that we can reformulate problem \eqref{eq:ted} (where $\bm{C}$ is replaced by $\tilde{\bm{C}}$) as the parametric Volterra integral equation
$
\tilde{ \bm{u} }(t, \bm{z} )= \bm{u} _0 + \int_{0}^{t}  \bm{h} (t,s,\tilde{ \bm{u} }(s, \bm{z} ), \bm{z} ) ds$, $t \in [0,T],
$
where $ \bm{h} (t,s, \bm{u} , \bm{z} )= \tilde{\bm{M}} (t, \bm{z} ) \bm{u} $ and $ \tilde{\bm{M}} (t, \bm{z} )= \tilde{\bm{C}}( t,\bm{z}) \odot  \bm{W} (t) -  \bm{D} ( \tilde{\bm{C}}( t,\bm{z}), \bm{W}(t))$, with $\bm{D}$ defined as in equation \eqref{eq:D}.

For any $ \bm{z}  \in \mathbb{C}^d$, we consider the solution $\tilde{ \bm{u} }(t, \bm{z} )$ be the solution for the diffusion equation \eqref{eq:ted} on a graph associated with the matrix $\tilde{\bm{C}}(t,\bm{y})$ (having complex-valued entries). The existence and uniqueness of $\tilde{\bm{u}}(\cdot, \bm{z})$ is guaranteed by Lemma \ref{lem:holo}, whose conditions will be checked below. Then, we simply define $\tilde{\bm{f}}( \bm{z} ) = \tilde{ \bm{u} }(T, \bm{z} )$. Our goal is to show that $ \bm{z} \mapsto\tilde{ \bm{u} }(t, \bm{z} )$ is holomorphic in $\mathring{\mathcal{K}}$ for all $t\in [0,T]$. Furthermore, observe from Assumption \ref{ass:diffusion_matrix} that the entries of $ \tilde{\bm{C}}(t,\bm{z})$ are bounded from above (as $\tilde{\bm{C}}$ is continuous on a compact set). 
To apply Lemma \ref{lem:holo}, we need to verify that:
\begin{itemize}
\item $ \bm{g} (t, \bm{z})\equiv \bm{u}_0$ is continuous in $[0,T] \times \mathcal{K}$, and $ \bm{h} (t,s, \bm{u} , \bm{z} )= \tilde{\bm{M}} (t, \bm{z} ) \bm{u} $ is continuous in $[0,T]^2 \times  \mathbb{C}^{|V|}  \times \mathcal{K}  $;
\item the parametric map $ \bm{z} \mapsto \bm{g} (t, \bm{z} )$ is holomorphic for any fixed $t\in[0,T]$,
\item the parametric map $(\bm{v},\bm{z}) \mapsto \bm{h}(t,s,\bm{v},\bm{z})$ is holomorphic in $\mathbb{C}^{|V|} \times \mathcal{K}$,
\item $\norm{ \bm{h} (t,s, \bm{u} , \bm{z} ) -  \bm{h} (t,s, \bm{v} , \bm{z} )} \leq L\norm{ \bm{u} - \bm{v} }$ for some Lipschitz constant $L>0$, $\forall t,s\in[0,T]$, $\forall  \bm{u}, \bm{v} \in\C^{|V|}$, $\forall \bm{z} \in\mathcal{K}$. \rev{Here we will verify this condition for $\|\cdot\| = \|\cdot\|_\infty$.}
\end{itemize}

Firstly, $ \bm{g} (t, \bm{z} )\equiv  \bm{u}_0$ is a constant function, therefore $ \bm{g} $ is continuous in $[0,T]\times\mathcal{K}$ and holomorphic in $\mathring{\mathcal{K}}$ for all fixed $t\in [0,T]$. 

Next, we want to show that $ \bm{h} $ is continuous in $[0,T]^2 \times  \mathbb{C}^{|V|} \times  \mathcal{K}$, and holomorphic in $\mathring{\mathcal{K}}$. First recall that affine maps are holomorphic and that composition of holomorphic maps is holomorphic \cite{henrici1993applied}. The map $ \bm{z} \mapsto \tilde{\bm{C}}( t,\bm{z})$ is holomorphic by Assumption \ref{ass:diffusion_matrix}. For two generic matrices $ \bm{A} $ and $ \bm{B} $, the map $ \bm{B} \mapsto \bm{B} \odot \bm{A} $ is linear in $ \bm{B} $, therefore it is holomorphic. Furthermore, the composition of two holomorphic functions is holomorphic. Thus, the map $ \bm{z} \mapsto \tilde{\bm{C}}( t, \bm{z})\odot \bm{W} (t)$ obtained from the composition of holomorphic maps is also holomorphic. Moreover, the map $ \tilde{\bm{C}} \mapsto\delta_{ij}\sum_k \tilde{C}_{ik}W_{ik}$ is linear in $ \tilde{\bm{C}} $, so it is holomorphic. Thus, the map $ \bm{z} \mapsto \bm{D} ( \tilde{\bm{C}} ( \bm{z} ))$, which is the composition of $ \bm{z} \mapsto \tilde{\bm{C}} ( \bm{z} )$ and $ \tilde{\bm{C}} \mapsto\delta_{ij}\sum_k \tilde{C}_{ik}W_{ik}$, is also holomorphic. Finally, the difference of two holomorphic functions is holomorphic. It then follows that $ \bm{z} \mapsto \tilde{\bm{C}} (t, \bm{z})\odot \bm{W }(t) -  \bm{D} ( \tilde{\bm{C}} ( \bm{z} ))$ is holomorphic. This shows that $ h $ is holomorphic.

All that is left to verify is that the Lipschitz condition \eqref{eq:lemma_Lip_cond} holds for $ \bm{h} (t,s, \bm{u} , \bm{z} )= \tilde{\bm{M}} (t, \bm{z} ) \bm{u} $. We start by observing that
$\norm{ \bm{h} (t,s, \bm{u} , \bm{z} ) -  \bm{h} (t,s, \bm{v} , \bm{z} )} = \norm{ \tilde{\bm{M}} (t, \bm{z} )( \bm{u} - \bm{v} )} 
    \leq \norm{ \tilde{\bm{M}} (t, \bm{z} )}\cdot\norm{ \bm{u} - \bm{v} }$,
where the norm in $\norm{ \tilde{\bm{M}} (t, \bm{z} )}$ is the matrix norm induced by $\|\cdot\|$. Here, we choose $\norm{\cdot}=\norm{\cdot}_\infty$. 
We now need to find an upper bound to $\norm{ \tilde{\bm{M}} (t, \bm{z} )}_{\infty \rightarrow \infty} = \max_{i\in V}\norm{\mathrm{row}_i( \tilde{\bm{M}} (t, \bm{z} ))}_1$. We know that
\begin{align*}
    \norm{\mathrm{row}_i( \tilde{\bm{M}} (t, \bm{z} ))}_1 
    &= \sum_{j\in V}\abs{\tilde{M}_{ij}(t, \bm{z} )} 
    = \sum_{j\in V}\abs{\tilde{C}_{ij}(t, \bm{z} )W(t)_{ij} - \delta_{ij}\sum_{k\in V}\tilde{C}_{ik}(t, \bm{z} )W(t)_{ik}} \\
    &\leq \sum_{j\in V}\abs{\tilde{C}_{ij}(t, \bm{z} )}W(t)_{ij} + \sum_{k\in V}\abs{\tilde{C}_{ik}(t, \bm{z} )}W(t)_{ik} \\
    &= 2\cdot \sum_{j\in V}\abs{\tilde{C}_{ij}(t, \bm{z} )}W(t)_{ij}   
    \leq 2 \cdot \max_{j \in V}\abs{\tilde{C}_{ij}(t, \bm{z} )}\cdot \sum_{j\in V}W(t)_{ij} < \infty.
\end{align*}
Since $ \tilde{\bm{C}} $ is continuous in $(t, \bm{z} )$ and $[0,T]\times\mathcal{K}$ is compact, the image $ \tilde{\bm{C}} ([0,T]\times\mathcal{K})$ is also compact. Hence, the maximum in the last line is finite. The sum in the last line is also finite, because $|V|$ is finite and $\bm{W}: [0,T] \mapsto \R^{|V| \times |V|}$ is continuous on a compact set. 
Therefore, $\norm{ \bm{M} }_\infty$ is bounded, and the Lipschitz condition \eqref{eq:lemma_Lip_cond} holds with $L = \norm{ \bm{M} }_{\infty\rightarrow \infty }<\infty$.

\paragraph{Step 2: validity of the uniform bound \eqref{eq:unif_upper_bound_ext}} We want to bound $\|\tilde{\bm{f}}\|_{L^\infty(\mathcal{H})}$. Note that 
\begin{equation}
   \|\tilde{\bm{f}}\|_{L^\infty(\mathcal{H})}  = \sup_{\bm{z} \in \mathcal{H}} \| \tilde{\bm{f}}(\bm{z}) \|_2 
   = \sup_{\bm{z} \in \mathcal{H}} \| \tilde{\bm{u}}(T,\bm{z}) \|_2.
\end{equation}
Recall from Step 1 that  $\tilde{\bm{u}}= \tilde{\bm{u}}(\cdot, \bm{z})$ satisfies $\dot{\tilde{\bm{u}}} = \tilde{\bm{M}} \tilde{\bm{u}}$ with $\tilde{\bm{M}}(t,\bm{z}) = \tilde{\bm{C}}(t,\bm{z}) \odot \bm{W}(t) - \bm{D}(\tilde{\bm{C}}(t,\bm{z}))$.

Now, let $\tilde{\bm{v}}(t) :=  \sup_{\bm{z} \in \mathcal{H}} \bm{v}(t,\bm{z})$, where $ \bm{v}(t,\bm{z}) := \| \tilde{\bm{u}}(t,\bm{z}) \|_2^2 =  \langle \tilde{\bm{u}}(t,\bm{z}) , \tilde{\bm{u}}(t,\bm{z}) \rangle  $ and $\langle \bm{v}, \bm{w} \rangle = \bm{v}^* \bm{w}$ for any $\bm{v}, \bm{w} \in \mathbb{C}^{|V|}$. To make the notation lighter, we define $\dot{\tilde{\bm{u}}}(t,\bm{z}) = \frac{\partial \tilde{\bm{u}}(t,\bm{z})}{\partial t}$. We have that
\begin{align*}
    \frac{\partial \bm{v}(t,\bm{z})}{\partial t} 
    & =  \langle \tilde{\bm{u}}(t,\bm{z}), \dot{\tilde{\bm{u}}}(t,\bm{z})\rangle + \langle \dot{\tilde{\bm{u}}}(t,\bm{z}) ,  \tilde{\bm{u}}(t,\bm{z}) \rangle \\ 
    & = \langle \tilde{\bm{u}}, \tilde{\bm{M}} \tilde{\bm{u}} \rangle + \langle  \tilde{\bm{M}} \tilde{\bm{u}}, \tilde{\bm{u}} \rangle 
 = \langle \tilde{\bm{u}}, \tilde{\bm{M}} \tilde{\bm{u}} \rangle + \overline{\langle  \tilde{\bm{u}},   \tilde{\bm{M}} \tilde{\bm{u}} \rangle}  = 2 \cdot \text{Re} \langle  \tilde{\bm{u}}, \tilde{\bm{M}} \tilde{\bm{u}}  \rangle \\
&\leq 2 r(\tilde{\bm{M}}) \| \tilde{\bm{u}} \|_2^2 = 2 r(\tilde{\bm{M})} \bm{v}(t,\bm{z})
\end{align*}
where $r(\tilde{\bm{M}}) = \sup_{\bm{w}} \frac{|\langle \bm{w}, \tilde{\bm{M}} \bm{w} \rangle|}{\langle \bm{w}, \bm{w} \rangle }$ is the \textit{numerical radius} of $\tilde{\bm{M}}$ (see, e.g., \cite{horn2012matrix}). The last inequality follows from the relation
$\frac{\text{Re} \langle \tilde{\bm{u}}, \tilde{\bm{M}} \tilde{\bm{u}} \rangle}{\langle \tilde{\bm{u}}, \tilde{\bm{u}} \rangle} \leq \frac{ | \langle  \tilde{\bm{u}}, \tilde{\bm{M}} \tilde{\bm{u}} \rangle |}{\langle \tilde{\bm{u}}, \tilde{\bm{u}} \rangle} \leq r(\tilde{\bm{M}})$.
Moreover, by the Cauchy-Schwarz inequality, 
\begin{align}
    r(\tilde{\bm{M}}) = \sup_{\bm{w}} \frac{|\langle \bm{w}, \tilde{\bm{M}} \bm{w} \rangle|}{\langle \bm{w}, \bm{w} \rangle } \leq \sup_{\bm{w}} \frac{\| \tilde{\bm{M}}\bm{w} \|_2\|\bm{w} \|_2}{\| \bm{w} \|_2^2} = \|\tilde{\bm{M}} \|_{2\rightarrow 2},
\end{align}
where  $\|\tilde{\bm{M}}\|_{2 \rightarrow2}$ denotes the matrix norm induced by $\|\cdot\|_2$. 

Defining $\beta (t) := \sup_{\bm{z} \in \mathcal{H}} \| \tilde{\bm{M}}(t,\bm{z}) \|_{2\to 2}$, we see that 
$\frac{\partial \bm{v}(t,\bm{z})}{\partial t} \leq \beta (t) \bm{v}(t,\bm{z})$,
which, in turn, leads to $\tilde{\bm{v}}'(t) \leq \beta(t) \tilde{\bm{v}}(t)$.
We can now apply Gronwall's lemma \cite{gronwall1919note}, obtaining
\begin{equation*}
\tilde{\bm{v}}(t) \leq \tilde{\bm{v}}(0) \cdot \exp \left(2\cdot \int_0^T \sup_{ \bm{z}  \in \mathcal{H}} \| \tilde{\bm{M}} (t, \bm{z} )\|_{2\to 2} \, d t\right),
\end{equation*}
which corresponds to \eqref{eq:unif_upper_bound_ext}.
We conclude by showing that $\int_0^T  \beta(t) dt < \infty$. Indeed, we have that $\beta$ is continuous in $t$, as $\| \cdot \|_{2 \to 2}$ is continuous and $\tilde{\bm{M}}(t,\bm{z})$ is continuous. Therefore $ \beta$ in bounded on $[0,T]$ and the integral is finite. This concludes the proof.
\end{proof}

\noindent Note that this proof works for the (stationary) edge-dependent diffusion case, with the corresponding parametric map $ \bm{y} \mapsto \bm{C}(\bm{y} )$, considering it is a special case of the time-and-edge-dependent diffusion case.

\vspace{10pt}
From now on, we will turn our attention back to scalar-valued functions $f_v: \C^d \mapsto \C$, as we wants to find a suitable sparse polynomial approximation to scalar-valued parametric map defined in \eqref{eq:goal} using the methods reviewed in Section \ref{subsec:sparse_pol_approx_concepts}. Note that analogous results could be derived in the vector-valued fields (and therefore, for the whole parametric solution of the diffusion map $\bm{y}\mapsto \bm{u}(\cdot, \bm{y})$) (see Remark \ref{rem:vector_valued_f}). 

\subsection{Best $s$-term approximation rates}

\noindent Based on Theorem \ref{thm:mainthm}, we can infer results about the convergence rate of the best $s$-term approximation error (recall Definition~\ref{def:best_s-term}) for the solution map $f_v$ defined in equation \eqref{eq:goal}. \rev{In fact, by using the holomorphy of $f_v$ one can show that its coefficients with respect to Legendre or Chebyshev orthogonal polynomials exhibit an exponential decay. In turn, an application of Stechkin's inequality (see, e.g., \cite[Section 7.4]{dung2018hyperbolic}) allows one to turn an exponential coefficient decay bound into a super-algebraic best $s$-term approximation rate estimate. Therefore, this shows that holomorphic regularity implies approximate sparsity (or compressibility) with respect to these orthogonal polynomial families. An in-depth survey of this theory can be found in \cite[Chapter~3]{sparsepoly}. Here,} by combining \cite[Theorem~3.6]{sparsepoly} \rev{(see also \cite{cohen2015approximation,cohen2011analytic})} and Theorem \ref{thm:mainthm}, we obtain the following corollary. 

\begin{cor}[Best $s$-term approximation rates for parametric graph diffusion]
\label{cor:maincor}
Let $T>0$ and let the parametric diffusivity matrix $ \bm{C}:[0,T] \times \R^{|V| \times |V|}$  satisfy Assumption \ref{ass:diffusion_matrix}. Let the map $f_v:\mathcal{U}=[-1,1]^d\to\C$, be defined by $f_v( \bm{y} )=  u_v(T, \bm{y} )$, where $v\in V$ and $\bm{u}$ is the solution of problem \eqref{eq:ted}. Moreover, assume that the compact set $\mathcal{K}\subseteq\C^d $ of assumption  \ref{ass:diffusion_matrix} is such that $\mathcal{E}_{\bm{\rho}}\subset\mathring{\mathcal{K}}$ for some $\bm{\rho}>\bm{1}$; let $B(\bm{u}_0, T, \rev{\tilde{\bm{M}}},\rev{\mathcal{E}_{\bm{\rho}}})$ be the bound in equation \eqref{eq:bound_term} and let $ \bm{c} = (c_{\bm{\nu}})_{\bm{\nu} \in \N_0^d} $ be the sequence of coefficient of $f_v$ with respect to either the Chebyshev or Legendre basis, i.e., $f_v = \sum_{ \bm{\nu}  \in \mathbb{N}_0^d} c_{ \bm{\nu}}  \psi_{ \bm{\nu} }$. Then, $\forall s\in\N_0$, and $0 < p \leq q < \infty$ it holds
$$
\sigma_s( \bm{c} )_q \leq \frac{ B(\bm{u}_0, T, \rev{\tilde{\bm{M}}},\rev{\mathcal{E}_{\bm{\rho}}}) \cdot C(p,\bm{\rho},d)}{(s+1)^{\frac{1}{p}-\frac{1}{q}}},
$$
\rev{where the constant $C(p, \bm{\rho}, d) > 0$ is as in \cite[Theorem~3.6]{sparsepoly}}.
\end{cor}

This result implies that the best $s$-term approximation error $\sigma_s( \bm{c} )_q$ with respect to Legendre or Chebyshev polynomials decays algebraically fast in $s$. In particular, $f_v$ defined in equation \eqref{eq:goal} admits a holomorphic extension to some compact $\mathcal{K}$, therefore the best $s$-term approximation error is bounded from above as in equation \eqref{eq:bounderror}. The term $\norm{f_v}_{L^{\infty}(\mathcal{E}_{\bm{\rho}})}$ in equation \eqref{eq:bounderror} is bounded as in Theorem \ref{thm:mainthm}, with $\mathcal{H} = \mathcal{E}_{\bm{\rho}}$. \rev{A numerical illustration of Corollary~\ref{cor:maincor} will be provided in Figures~\ref{fig:coeff_best_s-term} and \ref{fig:time_coeff_best_s-term}.}

\rev{The main takeaway of Corollary~\ref{cor:maincor} is that the map $f_v$ is compressible with respect to Legendre and Chebyshev polynomials. This, in turn, implies that  sparse polynomial methods are a suitable option to approximate it. The next section will provide convergence guarantees for such methods.}

\subsection{Convergence guarantees for least squares and compressed sensing}\label{subsec:LS_CS}

\noindent Theorem \ref{thm:mainthm} will be now leveraged to show algebraic convergence rates for both least squares and compressed sensing (in particular, for weighted SR-LASSO decoders; recall \eqref{eq:wSR-LASSO}). \rev{In particular, the error bounds proved in Corollaries~\ref{cor:LS} and \ref{cor:CS} are ``quasi-optimal'' in the sense that they follow the same decay rate for the best $s$-term approximation error from Corollary~\ref{cor:maincor} (with $q = 2$), i.e., $s^{1/2-1/p}$, after replacing the sparsity $s$ with the number of samples $m$ and up to log factors. Notably, these bounds show that least squares and compressed sensing can converge faster than the Monte Carlo rate $m^{-1/2}$. This is possible thanks to the holomorphic regularity of $f_v$. We start by illustrating the least squares case.}

\begin{cor}[Convergence of least squares-based surrogates]\label{cor:LS}
Let $0<\epsilon < 1$, $\varrho$ be either the uniform or Chebyshev measure on $\mathcal{U}=[-1,1]^d$, $m \geq 3$ and $\bm{y}_1, \dots , \bm{y}_m $ independent and identically distributed drawn from the probability distribution $ \varrho$.  Let the map $f_v:\mathcal{U}=[-1,1]^d\to\C$ be defined by $f_v( \bm{y} )= u_v(T, \bm{y} )$, where $\bm{u}$ is the solution of equation \eqref{eq:ted} with the \rev{parametric} diffusivity \rev{matrix} $\bm{C}(t,\bm{y})$ satisfying Assumption \ref{ass:diffusion_matrix} where $\mathcal{E}_{\bm{\rho}}\subset\mathring{\mathcal{K}}$ for some $\bm{\rho}>\bm{1}$; let $\{\psi_{\bm{\nu}}\}$ be either the Chebyshev or Legendre basis. Then, there exists a set $S \subset \mathbb{N}_0^d$ of cardinality $|S| \leq \lceil m/\log(m/\epsilon) \rceil$ such that the following holds with probability at least $1 - \epsilon$.
    For any $\bm{n} \in \mathbb{C}^m$, the approximation $\hat{f_v}= \sum_{\bm{\nu} \in S} \hat{c}_{\bm{\nu}} \psi_{\bm{\nu}}$ that solves the least squares problem \eqref{eq:ls}
is unique and satisfies, for every $0< p < 1$,
    \begin{equation*}
        \| f_v - \hat{f_v} \|_{L_{\varrho}^2(\mathcal{U})} \leq \rev{B(\bm{u}_0, T, \tilde{\bm{M}}, \mathcal{E}_{\bm{\rho}}) \cdot }C(\bm{\rho}, p) \cdot (m/\log (m/\epsilon))^{\frac{1}{2}-\frac{1}{p}} + 2 \cdot \|\bm{n} \|_{\infty},
    \end{equation*}
\rev{where $B(\bm{u}_0, T, \tilde{\bm{M}}, \mathcal{E}_{\bm{\rho}})$ is as in \eqref{eq:bound_term}.}

\end{cor}

The proof of Corollary \ref{cor:LS}, based on \cite[Theorem~6.1]{adcock2022monte}, can be found in Appendix~\ref{app:proofCorLS}.

\vspace{10pt}

\noindent
Theorem \ref{thm:mainthm} also implies uniform and nonuniform guarantees of algebraic convergence (using Theorems 7.12 and 7.13 in \cite{sparsepoly}) for compressed sensing. In particular, we aim at finding the minimizer of the weighted SR-LASSO decoder defined in \eqref{eq:wSR-LASSO}.

\begin{cor}[Convergence of compressed sensing-based surrogates]
\label{cor:CS}
 Let the map $f_v:\mathcal{U}=[-1,1]^d\to\C$ be defined by $f_v( \bm{y} )= u_v(T, \bm{y} )$, where $\bm{u}$ is the solution to problem \eqref{eq:ted} with the \rev{parametric} diffusivity \rev{matrix} $\bm{C}(t,\bm{y})$ satisfying Assumption \ref{ass:diffusion_matrix}, where $\mathcal{E}_{\bm{\rho}}\subset\mathring{\mathcal{K}}$ for some $\bm{\rho}>\bm{1}$; let $ \bm{c} =(c_{\bm{\nu}})$ be the sequence of coefficients of $f_v$ with respect to either the Chebyshev or Legendre basis,
Let $m \in \mathbb{N}$ be the number of samples, drawing $\bm{y}_1, \dots, \bm{y}_m$ independently from the measure $\varrho$, and $0 < \epsilon < 1$. 

Let $\Tilde{L} =  \kappa \cdot \log (2m) \cdot \big [ \log (2m) \cdot \min \{ \log (m) + d, \log (2d) \cdot \log (2m) \} + \log (\epsilon^{-1}) \big ]$,
with $\kappa>0$ a universal constant, $\tilde{m} = \tilde{m}(m, d, \epsilon) = m / \Tilde{L}$ and $\Lambda = \Lambda_{n-1}^{HC}$ be the hyperbolic cross index set defined in \eqref{eq:HC} with $n = \lceil \tilde{m} \rceil$. Then the following holds with probability at least $1-\epsilon$. Any minimizer $\hat{ \bm{c} }$ of \eqref{eq:wSR-LASSO} with the intrinsic weights $ \bm{w}   = (w_{ \bm{\nu} })_{ \bm{\nu}  \in \Lambda}$ where $w_{ \bm{\nu} } = \| \psi_{ \bm{\nu} } \|_{L^{\infty} (\mathcal{U})}, \; \forall  \bm{\nu}  \in \Lambda$, and parameter $\lambda = 1/(8 \sqrt{\tilde{m}})$ satisfies, for every $0 < p < 1$,
\begin{equation}
\begin{split}
    \| f_v - \hat{f_v} \|_{L_{\varrho}^2 (\mathcal{U})} \leq & \; \kappa_1 \cdot B(\bm{u}_0, T, \rev{\tilde{\bm{M}}, \mathcal{E}_{\bm{\rho}}}) \cdot \big ( C_1 + C_2 \big ) \cdot \tilde{m}^{1/2 - 1/p} + \kappa_2 \cdot \|  \bm{n}  \|_2 / \sqrt{m}, \nonumber \\
    \| f_v - \hat{f_v} \|_{L^{\infty}(\mathcal{U})} \leq & \; \kappa_3 \cdot B(\bm{u}_0, T, \rev{\tilde{\bm{M}}, \mathcal{E}_{\bm{\rho}}}) \cdot  \big ( C_1 + C_2 \big ) \cdot \tilde{m}^{1 - 1/p} + \kappa_4 \cdot \sqrt{\tilde{m}/m} \cdot \|  \bm{n}  \|_2 ,
\end{split}
\end{equation}
where $\kappa_i>0$, $i=1,\dots, 4$ are universal constants, $\hat{f} = \sum_{\bm{\nu} \in \Lambda} \hat{c}_{ \bm{\nu} } \psi_{ \bm{\nu} }$, 
$B(\bm{u}_0, T, \rev{\tilde{\bm{M}}, \mathcal{E}_{\bm{\rho}}})$ is the constant defined in $\eqref{eq:bound_term}$, $C_1 = C( \bm{\rho} ,p,d)$ is as in \cite[Lemma~7.19]{sparsepoly} and $C_2 = C( \bm{\rho} ,p,d)$ is as in \cite[Lemma~3.13]{sparsepoly}.
\end{cor}

\begin{proof}
    The proof follow directly from  \cite[Theorem~7.12]{sparsepoly} and Theorem \ref{thm:mainthm}. Indeed, Theorem \cite[Theorem~7.12]{sparsepoly} holds when \cite[Assumption~2.3]{sparsepoly} holds. Now, we consider Theorem \ref{thm:mainthm} with the following setting: 
     $\mathcal{H} = \mathcal{E}_{\bm{\rho}}$, 
     $\mathcal{K} = \mathcal{E}_{\bm{\rho}'}$, where $\bm{\rho}' > \bm{\rho}$ (strict inequality applies to every component), and,
     consequently, $\mathcal{O} = \mathring{\mathcal{K}}$.
    Therefore, \cite[Assumption~2.3]{sparsepoly} holds and so does \cite[Theorem~7.12]{sparsepoly}. 
    \end{proof}

We conclude by discussing a few remarks on Corollaries~\ref{cor:LS} and \ref{cor:CS}.
    
    \begin{remark}[Replacing SR-LASSO with QBCP]\label{rem:wQCBP_extension}
    Corollary~\ref{cor:CS} also holds for weighted QCBP (see \eqref{eq:wqcbp}). Indeed, \cite[Theorem~7.12]{sparsepoly} is formulated for weighted SR-LASSO, an inspection of its proof reveals that it also holds for weighted QCBP, up to replacing $\|\bm{n}\|_2$ with the QCBP tuning parameter $\eta$ (see also \cite[Theorem~7.4]{sparsepoly}, which the proof of \cite[Theorem~7.12]{sparsepoly} crucially relies on).
 \end{remark}

\begin{remark}[From algebraic to exponential rates]\label{rem:uniform_guarantees}
    Analogous guarantees of algebraic convergence for uniform recovery, and with exponential decay rates, hold as well. To obtain them it is enough to combine Theorem \ref{thm:mainthm} with \cite[Theorem~7.13]{sparsepoly} and \cite[Theorem~7.14]{sparsepoly} respectively. It is worth observing that the corresponding exponential rates are not dimension independent, though.
\end{remark}

\begin{remark}[From scalar- to vector-valued solution maps]\label{rem:vector_valued_f}
Corollaries \ref{cor:LS} and \ref{cor:CS} provide convergence guarantees for the construction of surrogate models to the solution map $\bm{y}\mapsto u_v(T,\bm{C}(T, \bm{y}))$ defined in equation \eqref{eq:goal}.   Nevertheless, these results could be extended to the whole solution map instead of its restriction to a single node, invoking results from least squares and compressed sensing for vector-valued complex functions. Indeed, in the least squares setting, an analogous version of Corollary \ref{cor:LS} holds for $\bm{y}\mapsto \bm{u}(T,\bm{C}(T, \bm{y}))$ up to replacing the $\ell^2$ norm with the Frobenius norm, yielding the following new problem:
\begin{equation}\label{eq:LS_vv}
     \min \limits_{ \bm{Z} \in \mathbb{C}^{N\times |V| }} \|  \boldsymbol{\Psi}  \bm{Z}  - \bm{B}  \|_F^2.
\end{equation}
On the other hand, the compressed sensing scalar-valued SR-LASSO scheme can be extended to the  following minimization problem:
\begin{equation}\label{eq:SR-LASSO_vv}
    \min \limits_{ \bm{Z} \in \mathbb{C}^{N\times |V| }} \lambda \|  \bm{Z}  \|_{1,1, \bm{w} } + \|  \boldsymbol{\Psi}  \bm{Z}  - \bm{B}  \|_F,
\end{equation}
which falls under the category of vector-valued compressed sensing (and more generally, compressed sensing for Hilbert-valued functions); see also \cite[Section~8.2.5]{sparsepoly}.
\end{remark}

\section{Numerical experiments}\label{sec:numerics}

\noindent In this section, we evaluate the efficacy of least squares and compressed sensing for approximating the (real-valued) parametric diffusion map $f_v(\bm{y}) = u_v(T, \bm{C}(\bm{y}))$ defined in \eqref{eq:goal}. The primary objective of these experiments is to assess the accuracy of the method across various graph structures, sizes, and diffusion settings. As a performance metric, we use the \emph{Root Mean Square Error (RMSE)}, defined by
$
\textnormal{RMSE} = \sqrt{\sum_{i=1}^{m_{\text{test}}}\frac{1}{m_{\text{test}}}(f_v(\boldsymbol{y}_i^{\text{test}}) - \hat{f_v}(\boldsymbol{y}_i^{\text{test}}))^2},
$
where $f_v(\boldsymbol{y}_i^{\text{test}})$ is the actual value and $\hat{f_v}(\boldsymbol{y}_i^{\text{test}})$ is the predicted value. In each of the following experiments, the test set will consist of 1000 pairs $(\bm{y}_i^{\text{test}},f(\bm{y}_i^{\text{test}}))$ where the $\bm{y}_i$'s are randomly sampled from the uniform distribution on $\mathcal{U}=[-1,1]^d$.

In Section \ref{subsec:synthetic_graphs} we first present our experiments on synthetic random graphs, more specifically, generated using the SBM, in order to produce graphs with a community structure. \rev{We then increase the complexity of the case study by focusing on a time-and-edge-dependent diffusion problem in Section \ref{subsec:time-dep_exp}.} In Section \ref{subsec:real-life} we replicate some of the same experiments \rev{of Section \ref{subsec:synthetic_graphs}} on real-world graph-structured datasets from Twitter and Facebook. \rev{Finally, in Section \ref{subsec:comp_cost} we carry out a computational cost analysis of the proposed surrogates compared to full-order simulations.}
\noindent We implement the proposed method in Python using the \texttt{NetworkX} library for graph generation and manipulation \cite{hagberg2008exploring}, the \texttt{NumPy} library for matrix operations to solve the diffusion equation, and the \texttt{equadratures} \cite{equad} and \texttt{SciPy} libraries for the polynomial approximation. The Python code is available at \url{https://github.com/k-yoan/surrogate_graph_diffusion}.

The setup for the diffusion process will be consistent throughout all experiments: given a diffusion parameter $\boldsymbol{y}\in\mathcal{U}$, we compute the solution $u_v(T)\in \R$ to the edge-dependent diffusion problem \eqref{eq:goal} where the node $v=2$ and for time $T=1$. To obtain the coefficients of the polynomial approximation, we use least-squares and (weighted) QCBP. 

\rev{We conclude with a general remark regarding numerical comparisons between least squares and compressed sensing. While this section includes several experiments where approximations obtained using these two techniques are compared directly, it is important to recall that these two methods are designed to be used in different sampling regimes. Indeed, on the one hand least squares is only able to provide meaningful approximations in the oversampled regime. On the other hand, compressed sensing works in the undersampled regime, but it becomes suboptimal in the oversamples regime due to the presence of bias induced by sparse regularization. Therefore, the numerical experiments where both techniques are applied in different regimes should not be thought of as an algorithmic comparison aimed at establishing a ``winner'', but rather as qualitative illustrations  showcasing the behavior of these methods in different scenarios.}

\subsection{Synthetic graphs \rev{(edge-dependent case)}}\label{subsec:synthetic_graphs}

\rev{In this section we consider an edge-dependent diffusion, which is a special case of the time-and-edge-dependent diffusion case described by equation \eqref{eq:Ct} in Example~\ref{example:diffusion_matrix}, where $h_j(t)\equiv 1$ for $j \in [d]$. Namely, the matrix $ \bm{C}$ in equation \eqref{eq:ed} is of the form}
\begin{equation}
\label{eq:C_exp}
\rev{\bm{C}(t, \bm{y}) =
\begin{pmatrix}
    \frac{y_{\sigma(1,1)}+1}{2}\cdot \mathds{1} & \cdots & \frac{y_{\sigma(1,K)}+1}{2}\cdot \mathds{1}\\
    \vdots & \ddots & \vdots\\
    \frac{y_{\sigma(K,1)}+1}{2}\cdot \mathds{1} & \cdots & \frac{y_{\sigma(K,K)}+1}{2}\cdot \mathds{1}\\
\end{pmatrix}.}
\end{equation}
For simplicity, we generate graphs from the SBM such that one graph always has the same number of nodes per community, and the edge probabilities $p_{i,j}$ are always set to $1$ (when nodes $i$ and $j$ belong to the same community) or $0.04$ (when nodes $i$ and $j$ are in different communities).
In the subsequent sections, each experiment is repeated 20 times. The geometric mean of RMSE obtained from those rounds is shown in bold colored lines/curves, while the variance is shown with a shaded region around the mean, corresponding to the geometric standard deviation of the RMSE.

\rev{For the moment, we will focus on time-independent problems. This choice is made mainly for computational simplicity. In fact, our numerical experiments involve statistical simulations based on Monte Carlo-type experiments, where independent experiments are run multiple times. In the time independent case, computing samples is very fast since we rely on an explicit forumla that depends on the eigendecomposition of the graph Laplacian $\bm{L}$. On the other hand, in the case of time-dependent problems samples are computed via a much slower numerical integration. Specifically, although Sections~\ref{subsec:PoC}--\ref{subsec:number_of_nodes} focus on time-independent problems, Sections~\ref{subsec:time-dep_exp}--\ref{subsec:comp_cost} will include time-dependent problems.}

\subsubsection{Proof of concept}\label{subsec:PoC}

\rev{In order to numerically validate our theoretical findings, we start by studying the decay and summability properties of the polynomial coefficients of the parametric diffusion map $f_v(\bm{y})$. In particular, we compute its polynomial coefficients in a total degree set of order $n=10$ using least squares with a large number of samples (namely, $m=10000$) to guarantee very high accuracy. Then, we approximate $\sigma_s(\bm{c})_2$ of Definition~\ref{def:best_s-term} with the best $s$-term approximation error of the computed coefficient vector. The results are shown in Figure~\ref{fig:coeff_best_s-term}. We can see that the sorted absolute coefficients (Figure~\ref{fig:coeff_best_s-term}, left) and the best $s$-term approximation error exhibit a super-algebraic decay. This is due to the holomorphic regularity of the map, proved in Theorem~\ref{thm:mainthm}, and is a direct consequence of \cite[Theorem~3.2]{sparsepoly} and Corollary~\ref{cor:maincor}, respectively.}

\begin{figure}[t]
    \centering
    \includegraphics[width=0.3\linewidth]{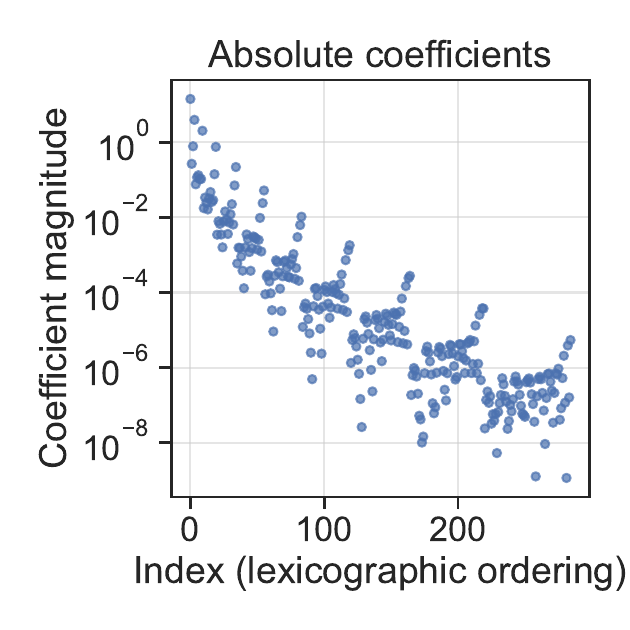} \includegraphics[width=0.3\linewidth]{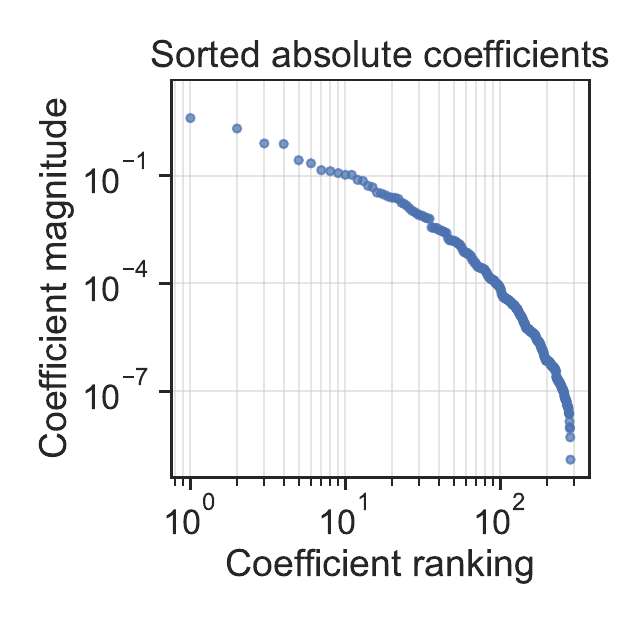}
    \includegraphics[width=0.3\linewidth]{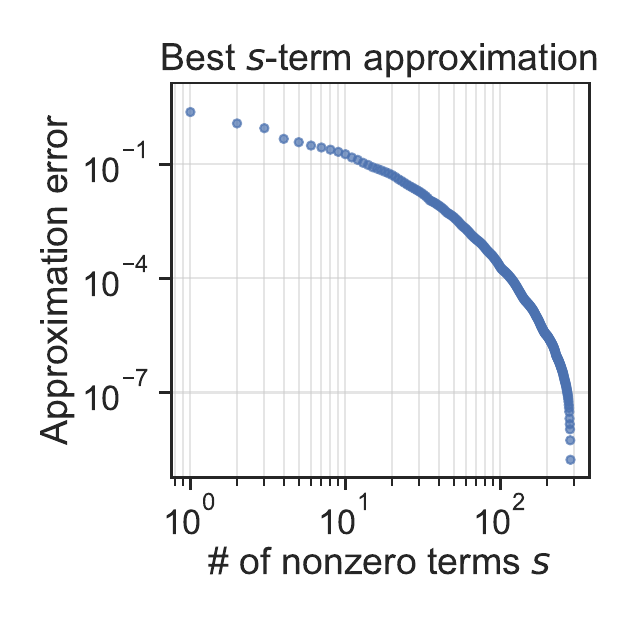}

    \caption{\rev{Absolute polynomial coefficients of the parametric diffusion map $f_y(\bm{y})$ (left: unsorted, center: sorted) and best $s$-term approximation error (right).}}
    \label{fig:coeff_best_s-term}
\end{figure}

\rev{Now}, we evaluate the accuracy of least squares and compressed sensing over $m$ sample points, for different dimensions and multi-index sets (total degree and hyperbolic cross, defined in \eqref{eq:TD} and \eqref{eq:HC}, respectively). We choose the orders $n$ such that the cardinality of the multi-index sets is around the same for both bases ($n=8$ for the total degree set and $n=20$ for the hyperbolic cross). Figure \ref{fig:rmse-vs-m-d3} shows that when $d=3$, we attain a high level of accuracy. 
\begin{figure}[t]
    \centering
    \includegraphics[width=0.45\linewidth]{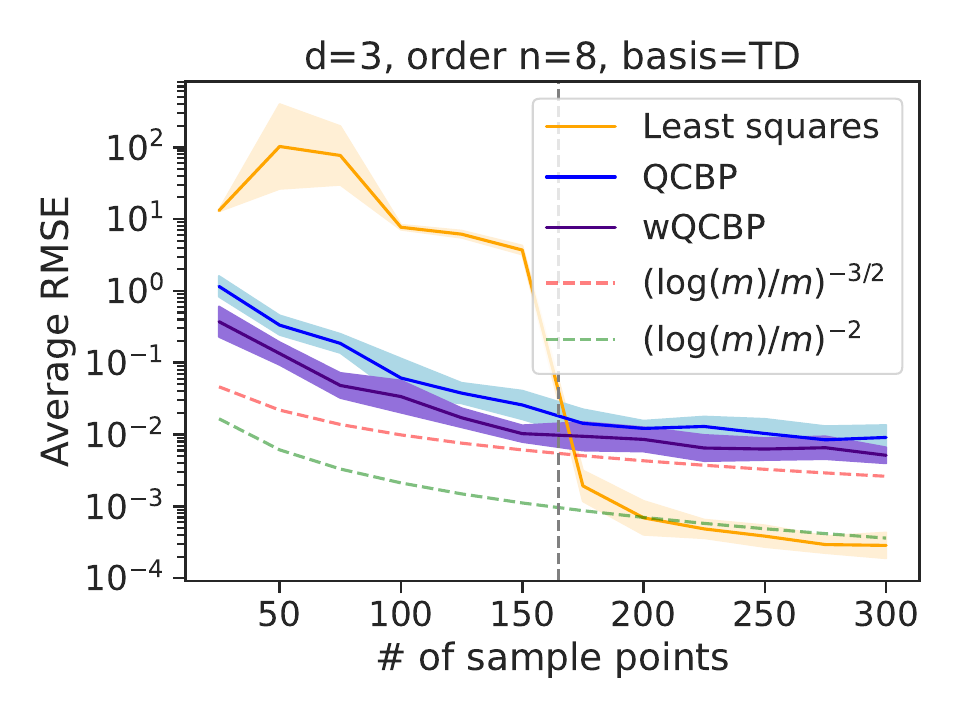} \includegraphics[width=0.45\linewidth]{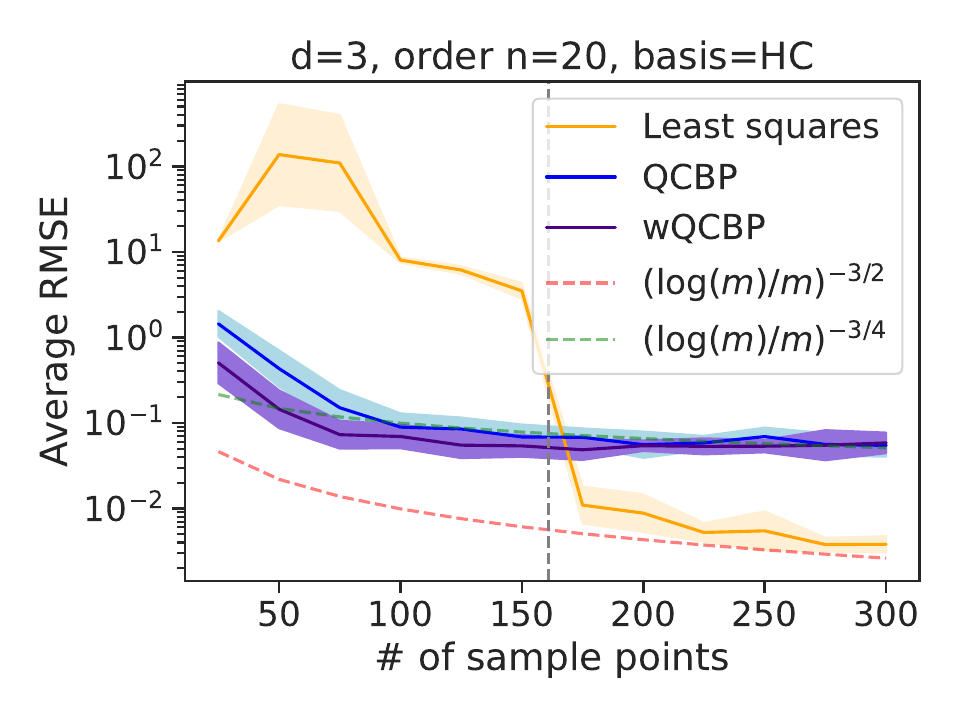}

    \caption{Convergence plot of average RMSE vs. number of sample points for $d=3$. \rev{The green and red dashed lines highlight the order of convergence of the methods. The vertical black dashed line represents the transition from the underdetermined to the overdetermined regime, i.e., $m=N$.} }
    \label{fig:rmse-vs-m-d3}
\end{figure}
Another thing to note is that when the number of sample points $m$ is less than the cardinality of the basis (represented by the vertical grey dashed line on both plots), QCBP and weighted QCBP yield better results than least squares. This is expected, as least squares only works when the system is overdetermined: as we increase the number of sample points, the average RMSE given by the least squares method decreases and starts performing better than the other two methods. \rev{In addition, we compare the observed convergence plots with curves of the form $(\log(m)/m)^{1/2-1/p}$ for different values of $p$ (see Corollaries~\ref{cor:LS} and \ref{cor:CS}). In all cases we find that the observed converges qualitatively exhibits this behavior for some value of $0< p< 1$. This illustrates that it is possible to numerically observe convergence rates of the form predicted by Corollaries~\ref{cor:LS} and \ref{cor:CS} (modulo a simpler log factor in the latter case), which are faster than the Monte Carlo rate $m^{-1/2}$.}

Figure \ref{fig:rmse-vs-m-d6} shows similar results with $d=6$. 
\begin{figure}[t]
    \centering
    \includegraphics[width=0.45\linewidth]{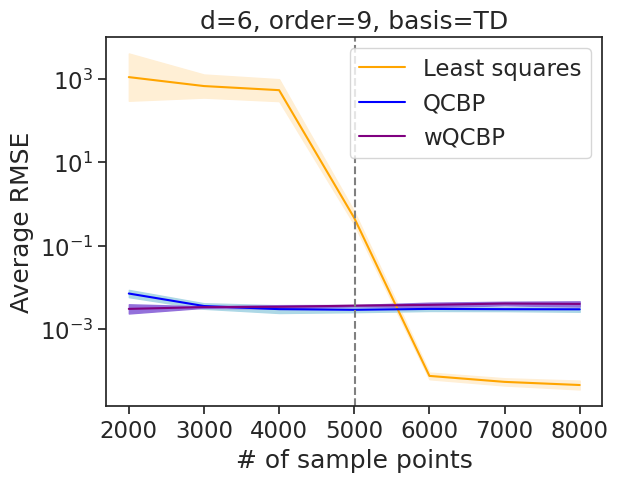}
\includegraphics[width=0.45\linewidth]{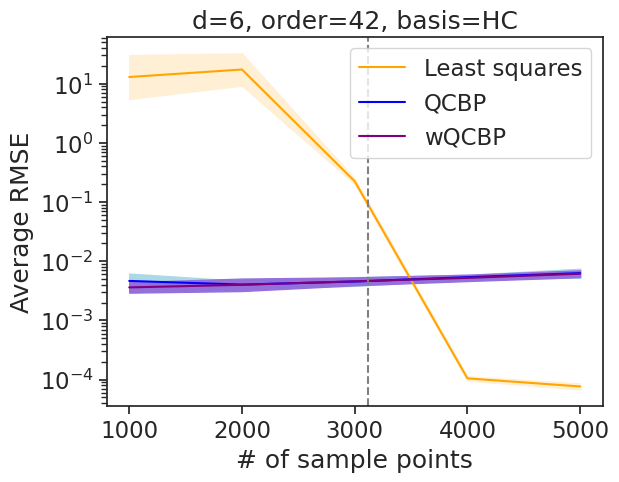}
    \caption{Convergence plot of average RMSE vs. number of sample points for $d=6$.}
    \label{fig:rmse-vs-m-d6}
\end{figure}
The way in which hyperparameters such as the multi-index set $\Lambda$, the dimension $d$ and the number of nodes $|V|$ affect the results is explored in the following subsections.

\subsubsection{Impact of cardinality of basis on accuracy}

\noindent In the previous experiments, we demonstrated the ability of least squares and compressed sensing to find accurate approximations to the solution of parametric diffusion problems, by evaluating the error as a function of the number of sample points. However, the accuracy of the method is influenced by several factors beyond just the number of sample points. One such factor is the cardinality of the multi-index set, which is influenced by the type of basis we choose, but also the order $n$ of that basis. Thus, we also explore how varying the size of the multi-index set affects the overall accuracy of our polynomial approximation.

In Figure \ref{fig:rmse-vs-card-m350}, we fix the dimension $d=3$ and the number of sample points $m=350$, for the total degree set and the hyperbolic cross. For both bases, we observe the same pattern: as we increase the cardinality of the multi-index set (i.e., the number of coefficients in the polynomial expansion), the average RMSE lowers drastically. The same experiment is replicated in Figure \ref{fig:rmse-vs-card-m1000} with $m=1000$, where we note the same correlation between the error and the size of the basis. The only major difference between the two experiments lies with the least-squares method: in Figure \ref{fig:rmse-vs-card-m350} (when $m=350$), the error increases sharply as the cardinality becomes large, whereas we only see a downward trend in Figure \ref{fig:rmse-vs-card-m1000}. 
\begin{figure}[t]
    \centering
    \includegraphics[width=0.45\linewidth]{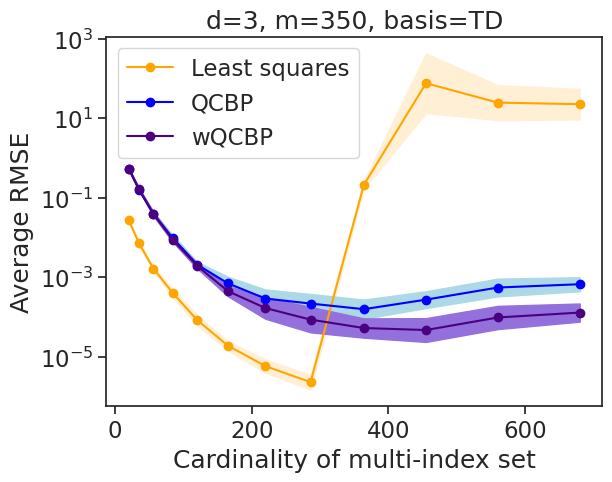}
\includegraphics[width=0.45\linewidth]{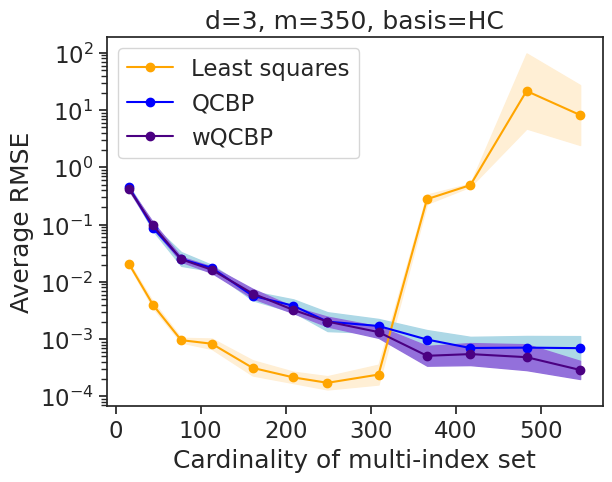}
    \caption{Convergence plot of average RMSE vs. cardinality of multi-index set for a fixed number of sample points $m=350$.}
    \label{fig:rmse-vs-card-m350}
\end{figure}
This is once again expected, as the least-squares method does not perform well when the cardinality of the set is greater than the number of sample points $m$.  
\begin{figure}[t]
    \centering
    \includegraphics[width=0.45\linewidth]{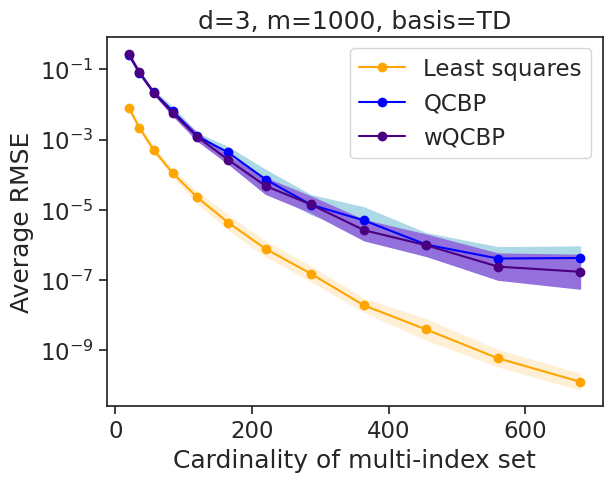}
\includegraphics[width=0.45\linewidth]{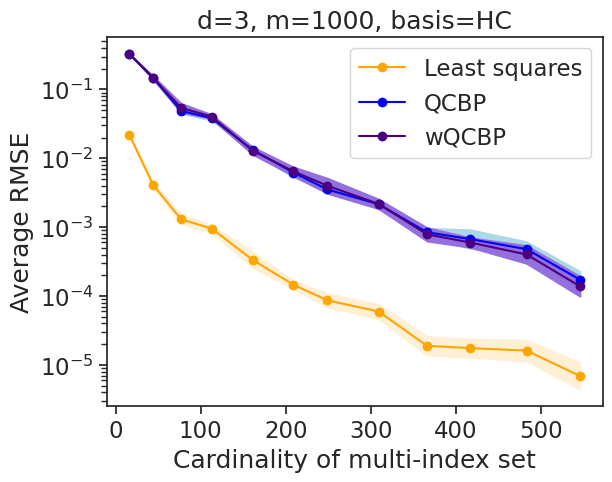}
    \caption{Convergence plot of average RMSE vs. cardinality of multi-index set for a fixed number of sample points $m=1000$.}
    \label{fig:rmse-vs-card-m1000}
\end{figure}


\subsubsection{Impact of dimension on a fixed size graph}

\noindent Next, we investigate the effect of varying the dimension $d$ (influenced by the number of communities $K$) on the accuracy of our polynomial approximation. We focus on a fixed-sized graph with 24 nodes, controlling for other factors such as the cardinality of the multi-index set, which was addressed in the previous subsection. Given the difficulty in finding an order $n$ that results in exactly the same cardinality for one basis across different dimensions, we instead chose orders that give the closest possible match in cardinality.

In the experiments, we test for both the total degree set and the hyperbolic cross. For $d=3$, $d=6$ and $d=10$, the cardinalities are 2925, 3003 and 3003, respectively, for the total degree. The results are shown in the left plots of Figures \ref{fig:24nodes-ls}, \ref{fig:temp2} and \ref{fig:temp3}.
\begin{figure}[t]
    \centering
    \includegraphics[width=0.45\linewidth]{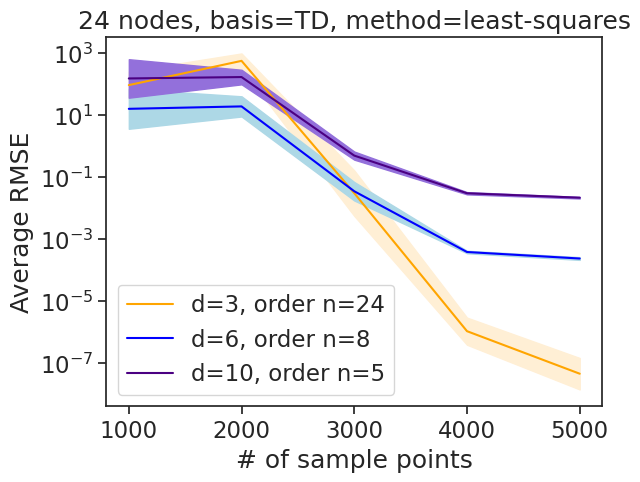}
    \includegraphics[width=0.45\linewidth]{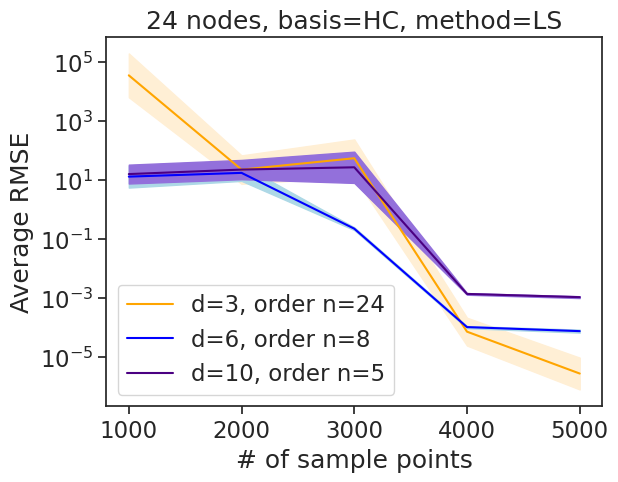}
    \caption{Convergence plots of average RMSE vs. number of sample points for a 24-node graph with varying number of communities (least squares).}
    \label{fig:24nodes-ls}
\end{figure}
\begin{figure}[t]
    \centering
    \includegraphics[width=0.45\linewidth]{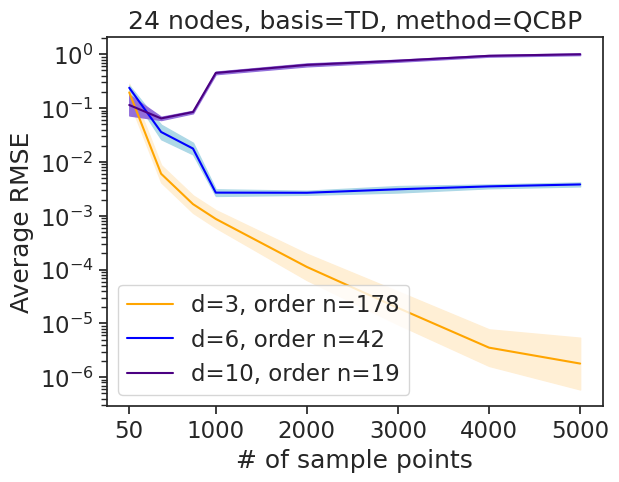}
    \includegraphics[width=0.45\linewidth]{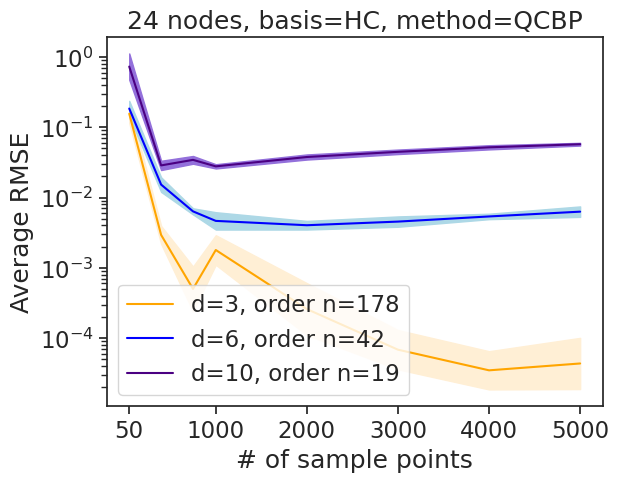}
    \caption{Convergence plots of average RMSE vs. number of sample points for a 24-node graph with varying number of communities (QCBP).}
    \label{fig:temp2}
\end{figure}
\begin{figure}[t]
    \centering
    \includegraphics[width=0.45\linewidth]{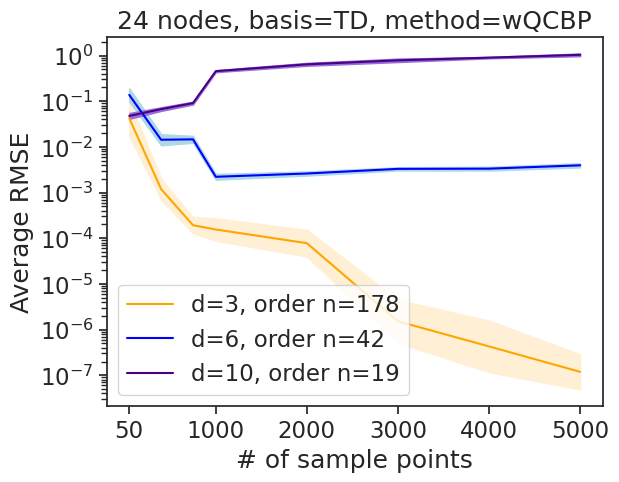}
    \includegraphics[width=0.45\linewidth]{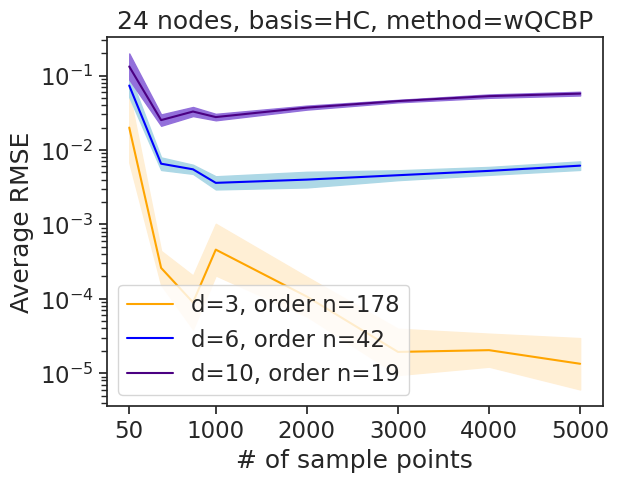}
    \caption{Convergence plots of average RMSE vs. number of sample points for a 24-node graph with varying number of communities (weighted QCBP).}
    \label{fig:temp3}
\end{figure}
Similarly, for the hyperbolic cross, the cardinalities are 3143, 3119 and 3076 (for $d=3$, $d=6$ and $d=10$ respectively). The results are shown in the right plots of Figures \ref{fig:24nodes-ls}, \ref{fig:temp2} and \ref{fig:temp3}. The overall pattern we observe from the three figures is that a lower dimension leads to a better accuracy.

A notable observation is that for higher dimensions, the hyperbolic cross seems to be better suited than the total degree. While the difference is not clearly evident for the least squares case in Figure \ref{fig:24nodes-ls}, the QCBP and weighted QCBP cases illustrated in Figures \ref{fig:temp2} and \ref{fig:temp3} show that for $d=10$ specifically, the results obtained with the hyperbolic cross are always superior to those with the total degree.

\subsubsection{Impact of number of nodes per community for fixed dimension}
\label{subsec:number_of_nodes}

\noindent In this subsection, we explore how the size $|V|$ of the graph influences the accuracy of our polynomial approximation. To isolate the effect of graph size, we keep the dimension $d$ fixed, by ensuring that the number of communities $K$ remains constant across experiments (we choose $K=2$, hence $d=3$). In this case, since the number of nodes per community is the factor that determines the overall graph size, we vary this quantity while maintaining a consistent configuration in terms of community structure. By adjusting the size of the graph in this way, we aim to understand how an increasing number of nodes within each community impacts the accuracy of polynomial approximation, helping determine how scalability influences its performance.

We once again plot the average RMSE against the number of sample points, for graphs with 5, 10 and 15 nodes per community. We only show results for $d=3$, with the total degree set as a basis and the least-squares method, as we would see a similar pattern for other multi-index sets and other methods. In Figure~\ref{fig:diff-size-graphs}, we can see that polynomial approximation yields better results when there are less nodes per community (i.e., when the graph is smaller). 
\begin{figure}[t]
    \centering
    \includegraphics[width=0.45\linewidth]{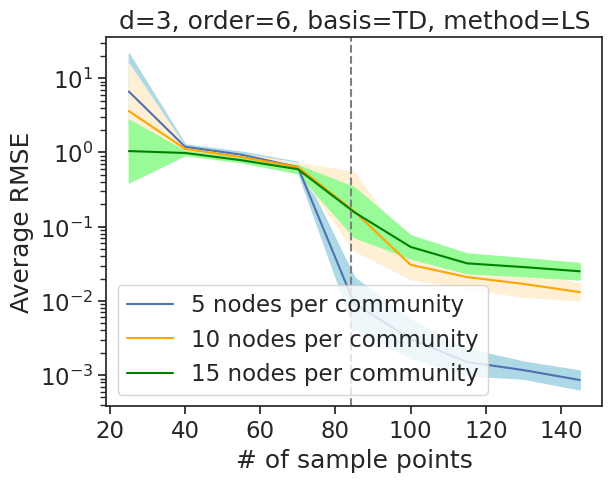}
    \includegraphics[width=0.45\linewidth]{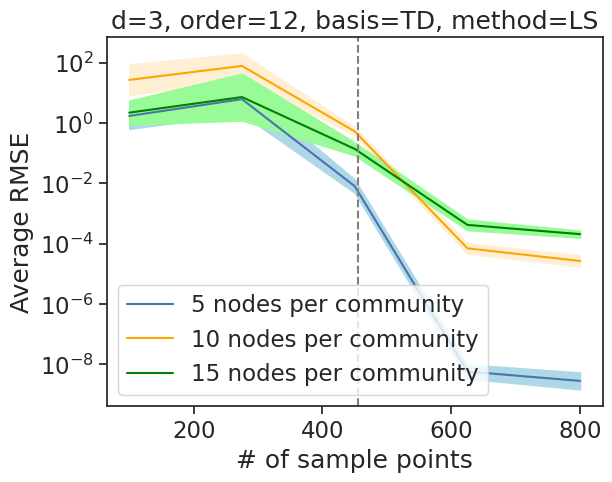}
    \caption{Convergence plots of average RMSE vs. number of sample points for graphs sampled from SBM with varying number of nodes (least squares).}
    \label{fig:diff-size-graphs}
\end{figure}
As we increase the number of sample points, the difference in performance seems to be more pronounced. The experiment was reproduced twice (with orders $n=6$ and $n=12$) to show that we can still achieve very good accuracy for larger graphs by simply increasing other hyperparameters, such as the order $n$.

\subsection{\rev{Synthetic graphs (time-and-edge-dependent case)}}
\label{subsec:time-dep_exp}

\rev{The aim of this section is to assess the performance of sparse polynomial surrogates  over a time-dependent version of the numerical study of Section~\ref{subsec:PoC}, reproducing the experiments in Figures~\ref{fig:coeff_best_s-term} and \ref{fig:rmse-vs-m-d3} in the time-and-edge-dependent case. In particular, we consider a diffusivity coefficient matrix $\mathbf{C}(t, \bm{y})$ as in equation~\eqref{eq:Ct} where 
\begin{align*}
    & h_{\sigma(1,1)}(t) = h_1(t) := \frac{1}{2}\left(1+\frac{t}{T}\right), \\
    & h_{\sigma(1,2)}(t) = h_{\sigma(2,1)}(t) = h_2(t) := \frac{1}{3}(2 + \cos (t)),\\
    & h_{\sigma(2,2)}(t) = h_3(t) := \frac{1}{3}(2 + \sin (t)).
\end{align*} 
These functions have been chosen in such a way to take strictly positive value less than or equal to $1$ for $t \in [0,T]$. 

We start by studying the polynomial coefficients and the best $s$-term approximation of the function $f_v(\bm{y})$. Figure~\ref{fig:time_coeff_best_s-term} shows the results of the same experiment as in Figure~\ref{fig:coeff_best_s-term} for the time-and-edge-dependent problem considered. Analogously to the edge-dependent case, we observe a clear super-algebraic decay of the sorted coefficients and of the best $s$-term approximation error. This numerically validates the holomorphic regularity of the underlying function established by Theorem~\ref{thm:mainthm}. 
In addition, we conduct the same experiment as in Figure~\ref{fig:rmse-vs-m-d3} in the time-and-edge-dependent case to study the convergence of the proposed surrogate models by plotting the average RMSE as a function of the number of samples $m$. The results, reported in Figure~\ref{fig:time_rmse-vs-m-d3}, show a convergence behavior very similar to that observed in the edge-dependent case.}
\begin{figure}[t]
    \centering
    \includegraphics[width=0.3\linewidth]{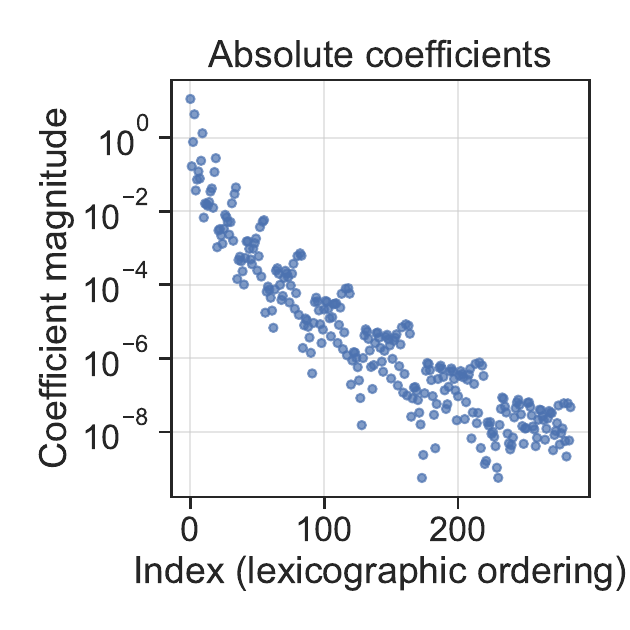} \includegraphics[width=0.3\linewidth]{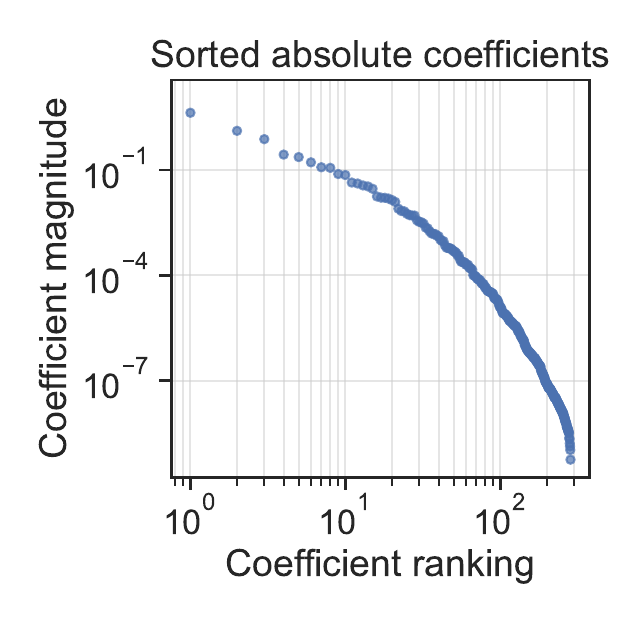}
    \includegraphics[width=0.3\linewidth]{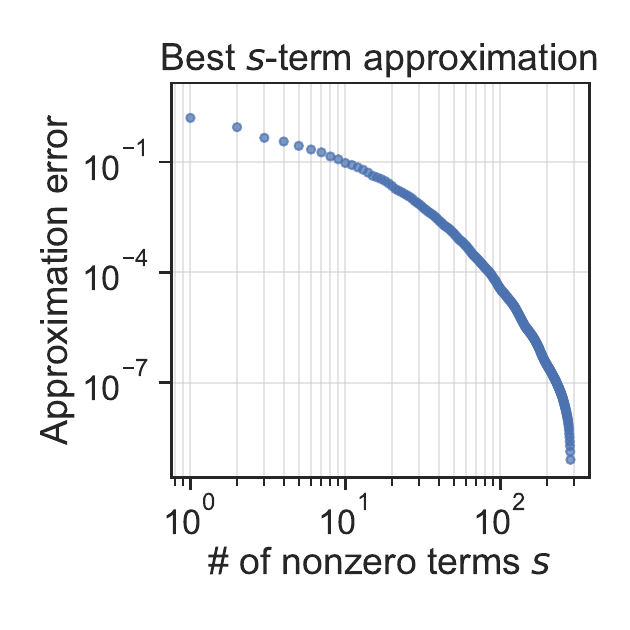}
    \caption{\rev{Time-and-edge-dependent case. Absolute polynomial coefficients of the parametric map $f_v(\bm{y})$ (left: unsorted, center: sorted) and best $s$-term approximation error (right).}}
    \label{fig:time_coeff_best_s-term}
\end{figure}
\begin{figure}[t]
    \centering
    \includegraphics[width=0.45\linewidth]{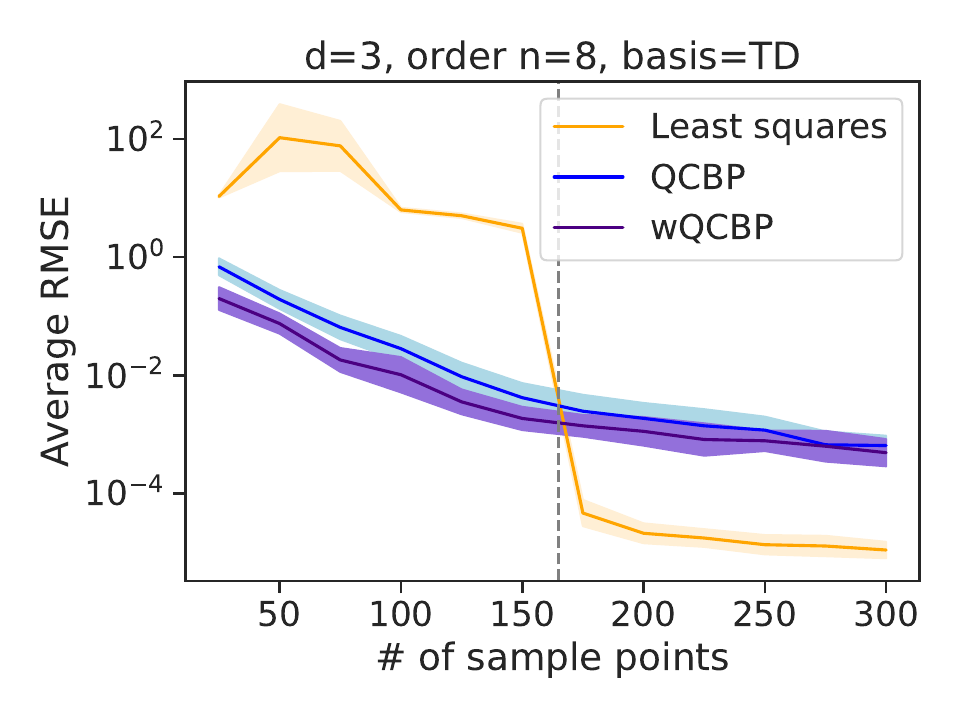} \includegraphics[width=0.45\linewidth]{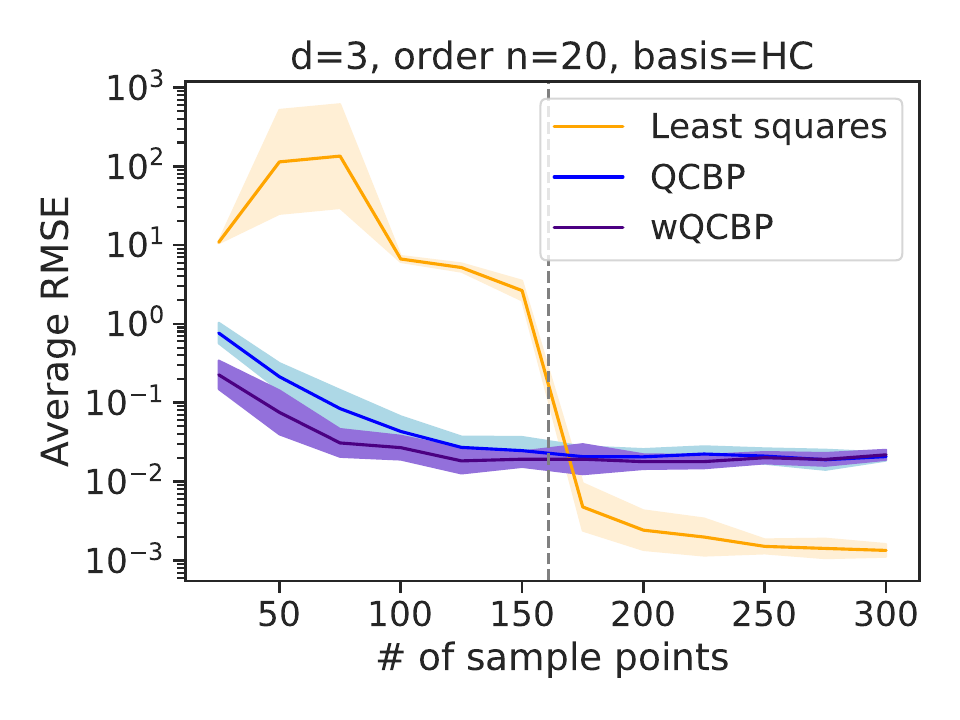}
    \caption{\rev{Time-and-edge-dependent case. Convergence plot of average RMSE vs.\ number of sample points for $d=3$.}}
    \label{fig:time_rmse-vs-m-d3}
\end{figure}
\rev{In summary, a comparison of the results in Figures~\ref{fig:time_coeff_best_s-term} and \ref{fig:time_rmse-vs-m-d3} with those in Figures~\ref{fig:coeff_best_s-term} and \ref{fig:rmse-vs-m-d3} shows that the introduction of a time-dependent diffusivity matrix does not numerically affect the performance of the surrogate method in the example considered.}

\subsection{Validation on real-world examples}\label{subsec:real-life}
\noindent To validate our theoretical results on real-world examples, we approximate $u_2$ (the solution of the diffusion equation at node 2) over datasets retrieved from Twitter and from Facebook, respectively. 
We find communities in each network using the fluid community detection algorithm \cite{pares2018fluid}.

\subsubsection{The Twitter dataset}
\label{subsec:twitter}
\noindent The Twitter network \cite{fink2023twitter},\cite{snapnets} represents the Twitter interaction network for the 117th United States Congress, both House of Representatives and Senate. The base data was collected via the Twitter’s API, then the empirical transmission probabilities (represented by weighted edges) were quantified according to the fraction of times one member retweeted, quote tweeted, replied to, or mentioned another member’s tweet.  
The number of nodes of the graph is $|V|=475$ and the number of edges is $|E| = 13,289$. We detected 2 communities exploiting the fluid community detection algorithm, which contain, respectively, 110 and 365 nodes. 
\begin{figure}[ht]
    \centering
    \includegraphics[height=5cm]{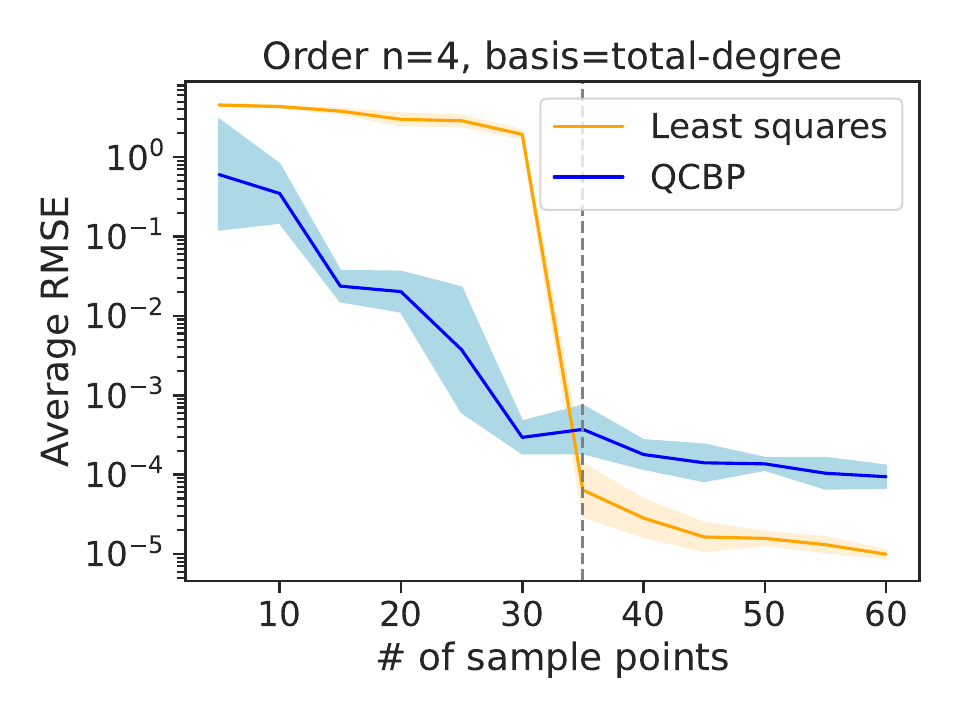}
   \includegraphics[height=5cm]{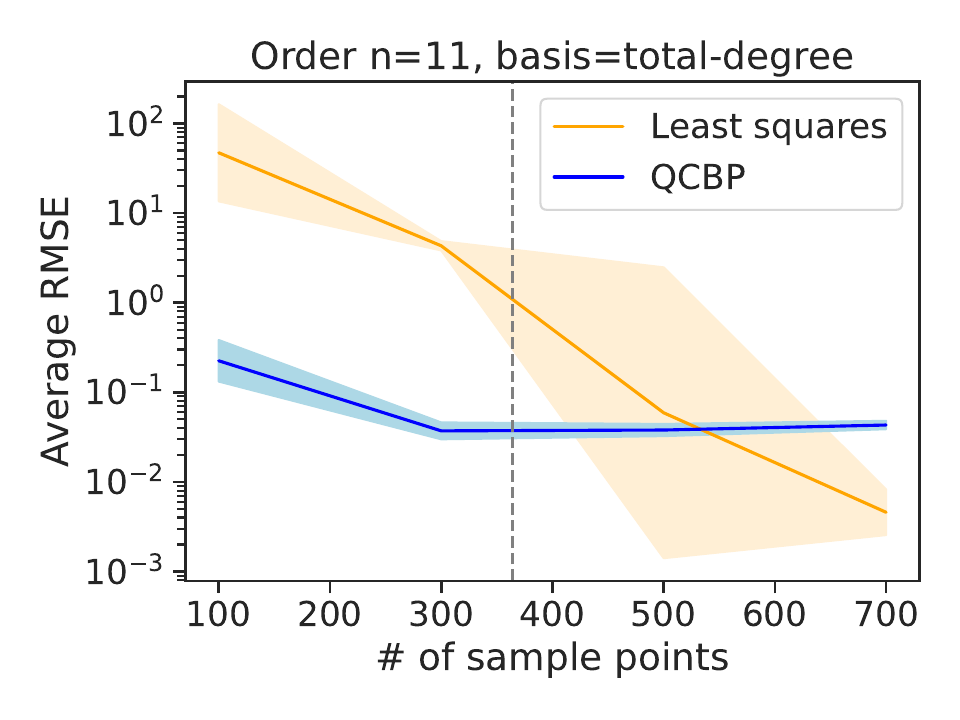}
    \caption{Left: Average RMSE on the approximation of the diffusion process over the Twitter dataset as a function of the number of samples. Right: Average RMSE on the approximation of the diffusion process over the Facebook dataset as a function of the number of samples.}
    \label{fig:Twitter}
\end{figure}

We computed the average RMSE on the approximation of $u_2$ over 5 runs, ranging the number of training sample points from 10 to 60 random sample points. The RMSE is computed over 1000 test random sample points. 
The experiments substantially confirm the trend already observed in Section \ref{subsec:PoC} of greater performance of least squares optimization, compared to the ones of QCBP, when the problem is overdetermined, and the ability of compressed sensing to recover good approximation in the underdetermined regime.

\subsubsection{The Facebook dataset}
\noindent This dataset consists of ``circles'' (or ``friends lists'') from Facebook. Facebook data was collected from survey participants using a specific Facebook app \cite{snapnets}\cite{leskovec2012learning}.
The size of this network is one order of magnitude larger than Twitter's. The number of nodes of the graph is $|V|=4,039$ and the number of edges is $|E| = 88,234$. We detected 2 communities employing the fluid community detection algorithm, which contain, respectively, 887 and 3152 nodes. 

We computed the average RMSE on the approximation of $u_2$ over 5 runs, ranging the number of training sample points from 100 to 700 random sample points. The RMSE is computed over 1000 test random sample points. Although performances in terms of RMSE loss are worse, due to the greater size of the graph, the qualitative observations made in Section~\ref{subsec:twitter} remain valid for the Facebook network.

\subsection{\rev{Computational cost analysis}}
\label{subsec:comp_cost}

\rev{In this section we conduct an experimental analysis of the computational cost of the proposed surrogate modeling strategy in order to showcase its applicability. Here we consider the time-and-edge-dependent problem of Section~\ref{subsec:time-dep_exp} and the real-world graphs of Section~\ref{subsec:real-life}. We denote as $t_{\mathrm{full}}$ the full order model time, i.e., the time needed to collect one sample by solving a diffusion equation. The polynomial surrogate has an offline and an online cost. The offline cost can be split into two contributions: namely, the cost of collecting $m$ samples---given by $m \cdot t_{\mathrm{full}}$---and the time $t_{\mathrm{coeff}}(\Lambda, m)$ needed to compute the polynomial coefficients. Finally, we denote  as $t_{\mathrm{online}}(\Lambda)$ the online time for the surrogate model, i.e., the time needed to evaluate the compute polynomial at one point.
In Figure \ref{fig:comp_cost_plots} (left) we plot the average times  of the offline and online cost of both surrogate and full order models (i.e., $t_{\mathrm{full}}$, $t_{\mathrm{coeff}}(\Lambda, m)$, and $t_{\mathrm{online}}(\Lambda)$) as a function of the number of samples $m$. Geometric means are computed over different random seeds: namely, 20 different seeds for the time-and-edge-dependent experiments on SBM, and 5 seeds for the edge-dependent case on real-world datasets.}
\begin{figure}[ht!]
    \centering
    \includegraphics[width=0.85\linewidth]{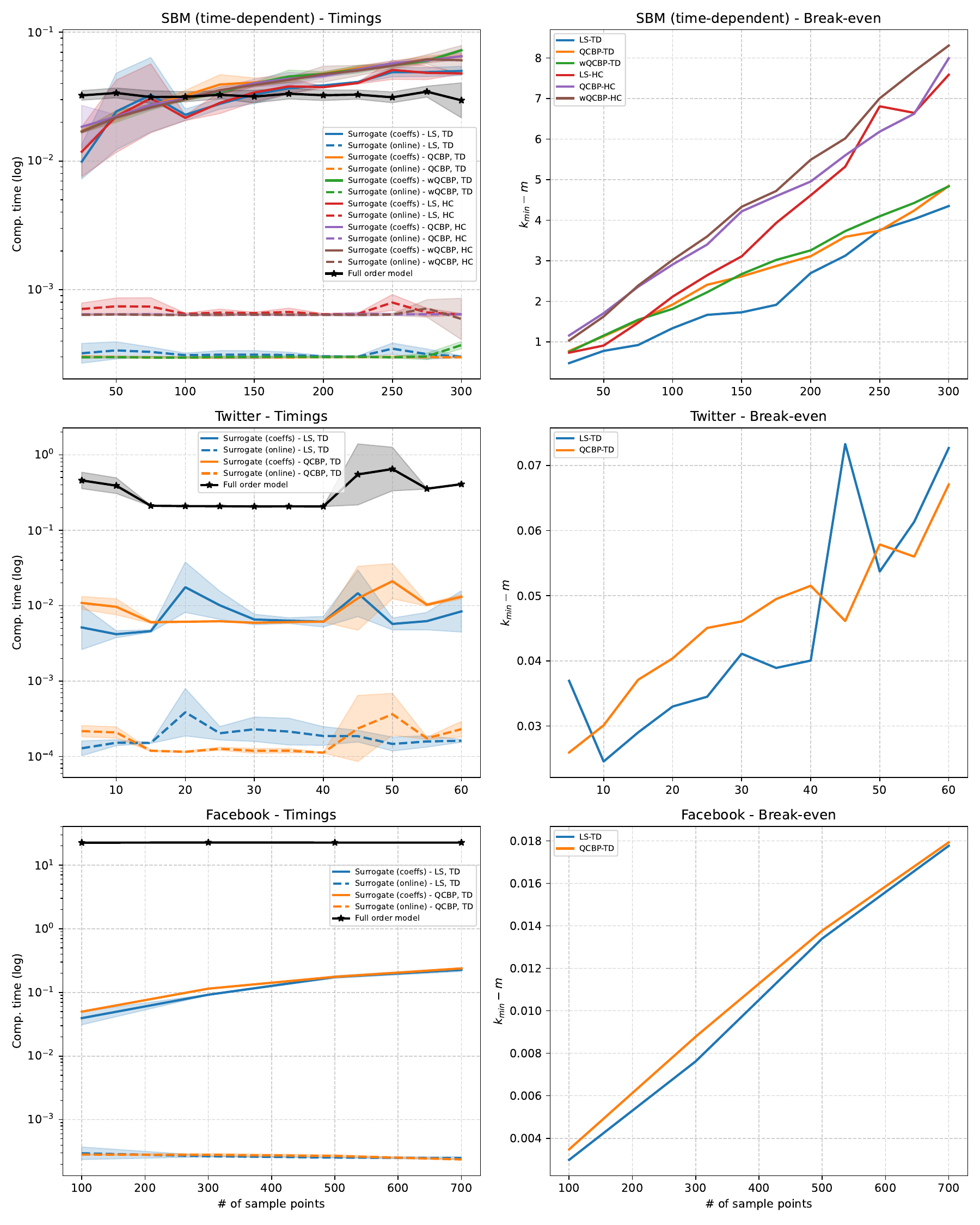}
    \caption{\rev{Left: Comparison of the average computational cost for different case studies. Solid lines refer to the cost $t_{\mathrm{coeff}}(\Lambda, m)$ of computing the coefficients of the surrogate model. The dashed lines represent the evaluation time $t_{\mathrm{online}}(\Lambda)$ of the surrogate at a single parameter vector. Solid lines with star markers refer to the evaluation time $t_{\mathrm{full}}$ of the full order model. Right: Plot of $k_{\min}-m$ (see \eqref{eq:def_kmin}) against the number of samples $m$. Top: time-and-edge-dependent diffusion on the SBM model. Center: edge-dependent diffusion on the Twitter dataset. Bottom: edge-dependent diffusion on the Facebook dataset.}}
    \label{fig:comp_cost_plots}
\end{figure}

\rev{In all cases, the online evaluation of the sparse polynomial surrogate is always faster than the computation of the full order model simulation by several order of magnitudes (i.e., $t_{\mathrm{online}}(\Lambda) \ll t_{\mathrm{full}}$). 
Moreover, while for real-world datasets the full order model evaluation of one sample  takes more time even than the offline computation of the coefficients of the surrogates (which makes the latter much better performing), for the time-dependent case the situation is different. In fact, we can see that computing the coefficients for compressed sensing surrogates with hyperbolic cross index set with order $s=20$ introduces a trade off, in terms of how many evaluations would be needed to make the computation of the surrogate worth to replace the full order model. 
For this reason,  we compute the minimum number of evaluations $k$ for which constructing the surrogate and evaluating it $k$ times is faster than evaluating the full order model $k$ times. This is defined by the inequality
$$
\underbrace{m \cdot t_{\mathrm{full}} + t_{\mathrm{coeff}}(\Lambda, m)}_{\text{surrogate construction (offline)}} + \underbrace{k \cdot t_{\mathrm{online}}(\Lambda)}_{k \text{ surrogate evaluations}} \leq \underbrace{k \cdot t_{\mathrm{full}}}_{k \text{ full order model evaluations}},
$$
which is solved by
\begin{equation}
\label{eq:def_kmin}
k \geq  \frac{m \cdot t_{\mathrm{full}} + t_{\mathrm{coeff}}(\Lambda, m)}{t_{\mathrm{full}} - t_{\mathrm{online}}(\Lambda)}=: k_{\min}.
\end{equation}
In Figure \ref{fig:comp_cost_plots} (right) we plot the difference $k_{\min}-m$ as a function of $m$. We plot $k_{\min}-m$ and not simply $k_{\min}$ to optimize visualization since in all cases, and especially for real-world datasets, we see that $k_{\min} \approx m$. Here $k_{\min}$ is averaged over multiple seeds as explained above. 
} 

\section{Conclusion}\label{sec:conclusion}

\noindent In this work, we have developed surrogate models based on sparse polynomial approximation for diffusion processes on graphs. To support the proposed methodology, we have shown convergence theorems for least squares and compressed sensing-based polynomial surrogates (see Corollaries~\ref{cor:LS} and \ref{cor:CS}, respectively), proving near-optimal algebraic rates as a function of the number of samples $m$. In Section \ref{sec:numerics}, we validated our theoretical findings by investigating the impact of single hyperparameters within the Stochastic Block Model (discussed in Section \ref{subsec:synthetic_graphs}) and showcasing the performance of least squares and compressed sensing-based sparse polynomial approximations on real-life datasets (Section \ref{subsec:real-life}).

Looking ahead, we identify several directions for future work. First, we aim to extend our numerical results for diffusion on graphs to cases involving time-dependent diffusion \rev{on further real-world graphs beyond the Twitter and Facebook datasets considered in this paper. In particular, we anticipate the benefit of the proposed surrogate modeling strategy to be most effective for very large-scale graphs.} 
Additionally, our analysis could be broadened to explore alternative parametrizations of the problem, such as focusing on initial conditions rather than just the diffusion coefficient matrix.

Beyond the diffusion equation, it would be valuable to develop polynomial-based surrogate models for more complex dynamic models on graphs. A notable example is the \emph{Kuramoto model} \cite{rodrigues2016kuramoto}, which describes how groups of oscillators tend to synchronize based on their network structure.

The insights provided by sparse polynomial surrogate models may also contribute to the characterization of deep learning surrogate models through the emerging field of \emph{practical existence theorems} \cite{adcock2022deep,adcock2025near}. The theory developed in this work could help bridge the gap in the literature on practical existence theorems relating to the approximation of functions on graphs.

Finally, the features highlighted in our work regarding time-dependent graph diffusion mechanisms within communities offer valuable insights for modeling deep learning architectures, particularly \emph{Graph Neural Networks (GNNs)} \cite{scarselli2008graph,kipf2016semi}. In recent years, GNNs have evolved based on diffusion rules \cite{chamberlain2021beltrami,chamberlain2021grand,toppingunderstanding} that generalize the standard message-passing scheme \cite{gilmer2017neural}. This evolution aims to address known challenges in the graph learning community, including oversmoothing, oversquashing, and the information bottleneck problem. Sparse polynomial approximation can be leveraged in the context of graph diffusion to design innovative GNN architectures capable of capturing more complex dynamics in community-aware scenarios.

\section*{Acknowledgments}
\noindent The authors warmly thank Jason Bramburger and Iacopo Iacopini for fruitful discussions. 
GAD acknowledges the support provided by the European Union - NextGenerationEU, in the framework of the iNEST - Interconnected Nord-Est Innovation Ecosystem (iNEST ECS00000043 – CUP G93C22000610007) project. The views and opinions expressed are solely those of the authors and do not necessarily reflect those of the European Union, nor can the European Union be held responsible for them. GAD acknowledges the support of INdAM Gruppo Nazionale of Calcolo Scientifico (GNCS). SB acknowledges the support of NSERC through grant RGPIN-2020-06766 and FRQ - Nature et Technologies through grants 313276 and 359708.
\bibliographystyle{elsarticle-num}
\bibliography{references}

\appendix
\section{Supplementary information about multi-index sets}\label{app:lower_sets}
\noindent Consider the setup of Section \ref{sec:surrogate_diffusion}. When no \emph{a priori} assumption is given on the multi-index set $S$,a  possible strategy is to choose a set $\Lambda\subseteq\N_0^d$ that contains $S$.
This could be achieved by taking the union of all subsets in $\N^d_0$ of cardinality less than or equal to $s$. However, with that sparsity condition only, the union would be equal to the whole set $\N^d_0$. Therefore, we need to consider sets $S$ with additional structure to ensure that we obtain a finite set.

\begin{defn}[Lower sets]
A multi-index set $\Lambda\subseteq\N_0^d$ is \emph{lower} if the following holds for every $ \bm{\nu} $,$ \bm{\mu} \in\N_0^d$:$
( \bm{\nu}  \in \Lambda \text { and }  \bm{\mu}  \leq  \bm{\nu} ) \Longrightarrow  \bm{\mu}  \in \Lambda,
$
where the inequality is understood componentwise.
\end{defn}
\smallskip

\noindent Intuitively, lower sets do not have any ``holes''. Note that the sets illustrated in Figure \ref{fig:indexsets} are lower sets, and for example, $S=\{(0,0), (2,0)\}$ is not a lower set. Lower sets are discussed in more detail in, e.g., \cite{sparsepoly}, but an important property is that a lower set $S$ of fixed size $s$ cannot contain multi-indices that are too far from the origin. In other words, this additional structure ensures that we obtain a finite set by taking the union of lower sets of cardinality less than or equal to $s$.

\begin{remark}\label{rem:lower_sets}
\noindent The union of lower sets with cardinality $s$ is a hyperbolic cross $\Lambda_{s-1}^{\mathrm{HC}}$ of order $s-1$ (see \cite[Proposition 2.5]{sparsepoly}). This emphasizes the connection between the hyperbolic cross index set and lower sets, and the importance of the hyperbolic cross in relation to sparse polynomial approximation, when the set $S$ is unknown.
\end{remark}
\section{Proof of Corollary \ref{cor:LS}}\label{app:proofCorLS} 
\noindent To prove Corollary \ref{cor:LS} we need an auxiliary result. Specifically, we need an adaptation of  \cite[Theorem 6.1]{adcock2022monte} to the finite-dimensional case. 

\begin{theorem}\label{th:main_ls}
    Let $0<\epsilon < 1$, $\varrho$ be either the uniform or Chebyshev measure on $\mathcal{U}=[-1,1]^d$, $m \geq 3$ and $\bm{y}_1, \dots , \bm{y}_m \sim _{\text{i.i.d}} \varrho$. Then, there exists a set $S \subseteq \mathbb{N}_0^d$ of cardinality $|S| \leq \lceil m/\log(m/\epsilon) \rceil$ such that the following holds with probability at least $1 - \epsilon$ for each fixed $f: \mathcal{U} \rightarrow \mathbb{C}$ that admits a holomorphic extension \rev{$\tilde{f}$} to a Bernstein polyellipse $\mathcal{E}_{\bm{\rho}} \subset \mathbb{C}^d$.
    For any $\bm{n} \in \mathbb{C}^m$, the approximation $\hat{f}$ to the problem \eqref{eq:ls}
    is unique and satisfies, for every $0< p < 1$, the bound $\| f - \hat{f} \|_{L_{\varrho}^2(\mathcal{U})} \leq C(\bm{\rho}, p) \rev{\cdot \|\tilde{f}\|_{L^\infty(\mathcal{E}_{\bm{\rho}})}}\cdot (m/\log (m/\epsilon))^{\frac{1}{2}-\frac{1}{p}} + 2 \cdot \|\bm{n} \|_{\infty}$.
\end{theorem}
\begin{proof}
    To prove the theorem it is enough to apply \cite[Theorem 6.1]{adcock2022monte} combined with \cite[Remark 5.2]{adcock2022monte}.
    We only need to show that the support $S$ of the best $s$-term approximation to $f$ satisfies $S \subset \mathbb{N}_0^d \times \{0\}^\mathbb{N}$. Let $\mathcal{F} = \{ \bm{\nu} \in \mathbb{N}_0^{\mathbb{N} }: \| \bm{\nu}  \|_0 < \infty \} \subset  \mathbb{N}_0^{\mathbb{N} }$.
Recall that $f(\bm{x}) = f(x_1, \ldots, x_d)$. Let $\bm{\nu} \in \mathcal{F}$ such that $\nu_k \neq 0$ for some $k>d$. Then
\begin{align*}
    c_{\bm{\nu}}  
    & = \int_{\mathcal{U}} f(\bm{x}) \psi_{\bm{\nu}}(\bm{x}) d\bm{x}
    = \int_{[-1,1]^d} f(\bm{x}) \prod_{j=1}^d \psi_{\nu_j}(x_j) \int_{[-1,1]^{\mathbb{N}}} \prod_{j>d} \psi_{\nu_j}(x_j) dx\\
    & = \int_{[-1,1]^d} f(\bm{x})  \prod_{j=1}^d \psi_{\nu_j}(x_j) dx_1\cdots dx_d \cdot
    \prod_{j>d, j \neq k}\int_{-1}^{1} \psi_{\nu_j}(x_j) d x_j
    \cdot 
    \underbrace{\int_{-1}^{1} \psi_{\nu_k}(x_k) dx_k}_{=0}  
    = 0
\end{align*}
Hence, the best $s$-term approximation support $S$ can be chosen in $\mathbb{N}_0^d \times \{0\}^\mathbb{N}$. 
\end{proof}

The proof of Corollary \ref{cor:LS} comes directly by combining Theorem~\ref{th:main_ls} with Theorem~\ref{thm:mainthm} considering the following setting: 
$\mathcal{H} = \mathcal{E}_{\bm{\rho}}$, 
$\mathcal{K} = \mathcal{E}_{\bm{\rho}'}$, where $\bm{\rho}' > \bm{\rho}$ (strict inequality applies to every component)
and, consequently, $\mathcal{O} = \mathring{\mathcal{K}}$.

\end{document}